\documentclass[10pt, reqno]{amsart}

\usepackage{amsmath,amssymb,amscd,mathrsfs,amscd}
\usepackage{graphics,verbatim,color,xcolor}

\theoremstyle{plain}
\newtheorem{theorem}{Theorem}[subsection]
\newtheorem{lemma}[theorem]{Lemma}
\newtheorem*{thm*}{Theorem}
\newtheorem*{cor*}{Corollary}
\newtheorem{proposition}[theorem]{Proposition}
\newtheorem{definition}[theorem]{Definition}
\newtheorem{corollary}[theorem]{Corollary}

\newtheorem{remark}[theorem]{Remark}

\numberwithin{equation}{section}

\newcommand{\af}{\alpha}

\newcommand{\ep}{\varepsilon}

\newcommand{\ld}{\lambda}
\newcommand{\sm}{\sigma}

\newcommand{\om}{\omega}

\newcommand{\la}{{\langle}}
\newcommand{\ra}{{\rangle}}

\newcommand{\set}[1]{\left\{#1\right\}}
\newcommand{\parens}[1]{\left(#1\right)}
\newcommand{\ang}[1]{\left\langle#1\right\rangle}
\newcommand{\bra}[1]{\left[#1\right]}

\newcommand{\bbinom}[2]{\begin{bmatrix}#1 \\ #2\end{bmatrix}}

\renewcommand{\bar}[1]{\overline{#1}}

\newcommand{\N}{{\mathbb{N}}}
\newcommand{\Z}{{\mathbb{Z}}}
\newcommand{\Q}{{\mathbb{Q}}}

\newcommand{\g}{{\mathfrak{g}}}

\renewcommand{\b}{{\mathfrak{b}}}

\newcommand{\height}{\mathrm{ht}}

\newcommand{\zero}{{\bar{0}}}
\newcommand{\one}{{\bar{1}}}

\newcommand{\tK}{{\tilde{K}}}
\newcommand{\tJ}{{\tilde{J}}}

\newcommand{\andeqn}{\,\,\,\,\,\, {\mbox{and}} \,\,\,\,\,\,}
\newcommand{\QED}{\rule{0.4em}{2ex}}

\newcommand{\ff}{{\mathbf f}}
\newcommand{\fint}{{{}_\A\mathbf f}}

\newcommand{\UU}{{\bf U}}
\newcommand{\Um}{{\UU^-}}
\newcommand{\Up}{{\UU^+}}
\newcommand{\UpJ}{{\UU^+_J}}
\newcommand{\Uz}{{\UU^0}}
\newcommand{\UzJ}{{\UU^0_J}}
\newcommand{\Uint}{{{}_\A\UU}}

\newcommand{\catO}{{\mathcal{O}}}

\newcommand{\A}{{\mathcal{A}}}
\newcommand{\Qqp}{{\Q^\pi(q)}}
\newcommand{\Zp}{{\Z^\pi}}

\newcommand{\ir}{{{}_i r}}
\newcommand{\ri}{{r_i}}

\newcommand{\bc}{{\mathbf{c}}}
\newcommand{\bh}{{\mathbf{h}}}
\newcommand{\bi}{{\mathbf{i}}}
\newcommand{\bj}{{\mathbf{j}}}

\input xy
\xyoption{all}

\begin{document}

\title{Quantum supergroups V. Braid group action}
\author{Sean Clark}
\address{Max Plank Institute for Mathematics,
53119 Bonn, Germany.}
\email{se.clark@mpim-bonn.mpg.de}

\author{David Hill }
\address{Department of Mathematics and Statistics, Washington State University - Vancouver, Vancouver, WA 98686, USA.}
 \email{david.hill@wsu.edu}

\date{\today}

\begin{abstract}
We construct a braid group action on quantum covering groups.
We further use this action to construct a PBW basis for the positive
half in finite type which is pairwise-orthogonal under the inner product.
This braid group action is induced by operators on the integrable
modules; however, these operators satisfy spin braid relations.
\end{abstract}

\maketitle

\section{Introduction}\label{S:Intro}

The action of the Weyl group $W$ on the Cartan subalgebra of a Kac-Moody algebra $\mathfrak g$ can be lifted 
to an action of the braid group $B_W$ on the enveloping algebra of $\mathfrak g$ and its integrable representations. 
Lusztig \cite{Lu,L88} generalized this construction to the quantum group $\UU_q(\mathfrak g)$
to give an action of $B_W$ on integrable
representations of $U_q(\g)$ via certain operators defined on each
weight space. Furthermore, these operators induce a compatible action
of the braid group on the quantized enveloping algebra itself.

This action of $B_W$ has been used by Lusztig \cite{L90}
to construct a family of PBW bases for the half-quantum group
when the associated Cartan datum is of finite type,
one for each reduced expression of the longest word in $W$.
This construction was generalized by Beck \cite{B}
to produce a convex PBW basis in affine type.
The action also has implications in the program of categorification,
where a (strong) categorical action of $\g$
induces a categorical action of $B_W$ on an associated category via
auto-equivalences \cite{CR,CKL,CK}.

The papers \cite{CHW1,CHW2,CFLW,C} introduced and studied
the properties of quantum covering groups $\UU=U_{q,\pi}(\g)$. These algebras
allow for the study of both Drin'feld-Jimbo quantum groups of Kac-Moody
Lie algebras alongside the quantum supergroup
associated to anisotropic Kac-Moody Lie superalgebras via the new
``half-parameter'' $\pi$ (first introduced in \cite{HW}), which satisfies $\pi^2=1$. Most of the structural
features of quantum groups have incarnations in the quantum covering groups;
for example, the quantum covering group admits a triangular decomposition
and the Chevalley generators satisfy higher Serre relations.
Additionally, the papers \cite{CHW2,CFLW} established the existence of a canonical
basis for quantum covering groups which specializes to the Lusztig-Kashiwara
canonical basis when $\pi=1$.

In this paper, we will construct a braid group action on the
quantum covering group $\UU$ using similar methods to \cite[Part V]{Lu}. 
In particular, we
first define certain operators on integrable $\UU$-modules. These operators generalize Lusztig's
construction, but come with additional factors of $\pi$ on each summand. The operators
are constructed by quantum exponentials of Chevalley generators, and in general
may not preserve the $\Z/2\Z$-grading of the modules. As a result, these operators
do not necessarily satisfy braid relations; rather, they satisfy {\em spin}
braid relations on isotypical components. In particular, though our approach to
this construction largely mimics Lusztig's, it often requires subtle and nontrivial
work to introduce the powers of $\pi$ in the various formulae.
Nevertheless, most of Lusztig's results admit analogues: these operators induce even automorphisms of $\UU$;
the automorphisms preserve the integral form of $\UU$; and they satisfy the braid
relations. As a result, we can construct a family of orthogonal PBW-type bases for the
quantum covering group associated to $\mathfrak{osp}(1|2n)$.

We note that in \cite{CHW3}, a family of PBW-type bases for $\UU_q(\mathfrak{osp}(1|2n))$
have been constructed via the combinatorics of Lyndon words.
We conjecture that these PBW bases should coincide with the PBW bases
constructed via braid operators whenever the reduced expression for the longest word is induced from a total ordering on the simple roots. We also conjecture that
PBW-type bases can be constructed in affine type using methods
similar to those in \cite{B}.

The paper is organized as follows.

In Section 2, we set notations
and recall some of the standard facts about quantum covering groups.

In Section 3, we introduce the braid group operators on integrable modules,
and deduce some basic properties. These operators are used to
construct automorphisms of $\UU$. Additionally, the interaction
between the braid operators and the coproduct are determined.

In Section 4, the braid automorphisms are considered as maps on
subspaces of the positive half-quantum group. In particular, it is shown
that the standard bilinear form is invariant under the braid
operators up to a factor of an integral power of $\pi$.

In Section 5, we show that the braid automorphisms of $\UU$ satisfy
the braid relations, whereas the braid operators on integrable modules
within certain blocks satisfy spin braid relations.
In particular, the braid automorphisms are used to produce
a PBW basis in finite type.
\vspace{1em}

\noindent\textbf{Acknowledgements.}
We would like to thank Weiqiang Wang for his interest in the paper and his helpful comments, as well as for the encouragement
to complete this project.

\section{Preliminaries}\label{S:Prelims}

In this section, we recall notation and results on quantum covering groups from \cite{CHW1}.

\subsection{Root data}\label{SS:Root Data}
Let $I=I_\zero \cup I_\one$ be a $\Z_2$-graded finite set of size $\ell$, for which we assume throughout that $I_\one\neq\emptyset$. Let $A=(a_{ij})_{i,j\in I}$ be a generalized Cartan matrix (GCM) such that
\begin{enumerate}
\item[(C1)] $a_{ii} =2$, for all $i\in I$;

\item[(C2)] $a_{ij} \in \Z_{\leq0}$, for $i\neq j \in I$;

\item[(C3)] $a_{ij}=0$ if and only if $a_{ji}=0$;

\item[(C4)] there exists an invertible matrix $D =\text{diag}(d_1,\ldots,
d_r)$ with $DA$ symmetric.
\end{enumerate}
We can and shall further assume $d_i \in \Z_{>0}$ and $\gcd
(d_1,\ldots, d_r)=1$.
We also define the symbols $b_{ij}=1-a_{ij}$.

Introduce the parity function $p(i)=0$ for $i\in I_\zero$ and $p(i)=1$ for $i\in I_\one$. Throughout the paper, we will impose the additional assumption:

\begin{enumerate}
\item[(P1)] $a_{ij} \in 2\Z$, for all $i\in I_\one$ and all $j\in I$;
\item[(P2)] for all $i\in I$, $d_i\equiv p(i)$ (mod 2).
\end{enumerate}

We note that (P2) is almost always satisfied for Cartan data of finite or affine type satisfying (P1).

Let $(P,P^\vee,\Pi,\Pi^\vee)$ be the root data associated to $A$. Here, $P$ and $P^\vee$ are free $\Z$-modules of rank $\ell$ (called the weight and coweight lattice, respectively). The simple roots (resp. coroots)
$$\Pi=\{\af_i|i\in I\}\subset P\;\;\; (\mbox{resp. }\Pi^\vee=\{\af_i^\vee|i\in I\}\subset P^\vee)$$
are linearly independent, and we define the root lattice
$$Q=\sum_{i\in I}\Z\af_i\;\;\;\mbox{and}\;\;\; Q_+=\sum_{i\in I}\Z_{\geq0}\af_i.$$
Furthermore, for $\nu=\sum \nu_i\alpha_i^\vee$ with $\nu_i\in \Z$, we define the notation
\begin{equation}\label{eq:deftilderoot}
\tilde{\nu}=\sum d_i\nu_i \alpha_i^\vee.
\end{equation}
We may define a $\Z_2$-grading on $Q$ by declaring $p(\af_i)=p(i)$ and extending linearly.
We also have a $\Z$-grading on $Q$ given by $\height(\sum_{i\in I} c_i\alpha_i)=\sum_{i\in I} c_i$.

Let
$$\ang{\cdot,\cdot}:P^\vee \times P\longrightarrow\Z$$
denote the perfect pairing defined by $\ang{\af_i^\vee,\af_j}=a_{ij}$, 
and let $\om_i\in P$ (resp. $\om_i^\vee\subset P^\vee$) 
be dual to $\af_i^\vee$ (resp. $\af_i$) with respect to this pairing. 
 We set $P_+=\set{\lambda\in P\mid \ang{\alpha_i^\vee,\lambda}\geq 0}$.

Also, define the symmetric bilinear form
$$(\cdot,\cdot):Q\times Q\longrightarrow \Z$$
by $(\af_i,\af_j)=d_ia_{ij}$. Observe that conditions (P1) and (P2) together imply that $(\mu,\nu)\in2\Z$  for any $\mu,\nu\in Q$,
hence in particular $\ang{\tilde\mu,\nu}\in 2\Z$ for any $\mu\in Q^\vee$ and $\nu\in Q$.

\subsection{The braid group and spin braid group}

The braid group $B=B(A)$ associated to a GCM $A$
is defined to be the group with generators $t_i$ ($i\in I$)
subject to the relations
\begin{align}\label{E:BraidRelations}
\underbrace{t_it_jt_i\cdots}_{m_{ij}}\,=\,\underbrace{t_jt_it_j\cdots}_{m_{ij}},
\end{align}
where the number of terms, $m_{ij}$, is determined by the product $a_{ij}a_{ji}$ as follows:

\begin{center}\begin{tabular}{c|ccccc}
$a_{ij}a_{ji}$&0&1&2&3&$\geq4$\\\hline
$m_{ij}$&2&3&4&6&$\infty$
\end{tabular}\end{center}

The braid group acts on $P$ and $P^\vee$ via simple reflections.
To wit, for $i\in I$, we define the simple reflection $s_i$, which acts on $P$ (resp. $P^\vee$) by the formula
$$s_i(\ld)=\ld-\ang{\af_i^\vee,\ld}\af_i,\;\;\;(\mbox{resp. }s_i(\ld^\vee)=\ld^\vee-\ang{\ld^\vee,\af_i}\af_i^\vee).$$
The Weyl group $W$ is the group generated by the set of reflections
$\{s_i|i\in I\}$. It is subject to the relations $s_i^2=1$ for $i\in I$ and the
braid relations \eqref{E:BraidRelations} (with $t_i,t_j$ replaced by $s_i,s_j$).

In addition to these standard definitions, we shall need a variant
of the braid group. We define the {\em spin} braid group $B^{\rm spin}=B^{\rm spin}(A,\varpi)$
associated to a GCM $A$ and parity function $\varpi:I\rightarrow \set{0,1}$ as follows.
Define the set of  $I_{\rm spin}\subset I\times I$ via
$I_{\rm spin}=\set{(i,j)\in I\times I\mid \varpi(i)=\varpi(j)=1,\mbox{ and } a_{ij}=0}$.
Then, $B^{\rm spin}$ is the group with generators $t_i$ ($i\in I$) and an additional
generator $\varsigma$ satisfying the following relations:
\begin{enumerate}
\item[(SB1)] $\varsigma^2=1$ and $\varsigma t_i=t_i\varsigma$ for all $i\in I$;
\item[(SB2)] if $(i,j)\notin I_{\rm spin}$,
$t_i$ and $t_j$ satisfy \eqref{E:BraidRelations};
\item[(SB3)] if $(i,j)\in I_{\rm spin}$, then $t_it_j=\varsigma t_jt_i$.
\end{enumerate}

\subsection{Parameters}

Let $q$ be a formal parameter and let $\pi$ be an indeterminate
such that
$$
\pi^2=1.
$$
For a ring $R$, we define $R^\pi=R[\pi]/(\pi^2-1)$.
We will work over (subrings of) the ring $\Qqp$.
This ring has idempotents
\begin{equation}\label{eq:pi idempotent}
\ep_{+}=\frac{1+ \pi}{2},\qquad\ep_{-}=\frac{1- \pi}{2},
\end{equation}
and note that $\Qqp=\Q(q)\ep_+\oplus \Q(q)\ep_-$.
In particular, since $\pi \ep_{\pm}=\pm \ep_{\pm}$
for any $\Qqp$-module $M$, we see that
\[M|_{\pi=\pm 1}\cong \ep_{\pm} M.\]

Let $\A=\Z^\pi[q,q^{-1}]$.
For $k \in \Z_{\ge 0}$ and $n\in \Z$,
we use a $(q,\pi)$-variant of quantum integers, quantum factorial and quantum binomial coefficients:

\begin{equation}
 \label{eq:nvpi}
\begin{split}
\bra{n}_{q,\pi} &
=\frac{(\pi q)^n-q^{-n}}{\pi q-q^{-1}}  \in \A,
  \\
\bra{n}_{q,\pi}^!  &= \prod_{l=1}^n \bra{l}_{q,\pi}   \in \A,
 \\
\bbinom{n}{k}_{q,\pi}
&=\frac{\prod_{l=n-k+1}^n  \big( (\pi q)^{l} -q^{-l} \big)}{\prod_{m=1}^k \big( (\pi q)^{m}- q^{-m} \big)}  \in \A.
\end{split}
\end{equation}

These $(q,\pi)$-quantum integers satisfy identities analogous to more traditional quantum integers.
\begin{align}
\label{eq:binomida} &\bbinom{a}{t}_{q,\pi}=(-1)^t\pi^{ta-\binom{t}{2}}\bbinom{t-a-1}{t}_{q,\pi}, \\
\label{eq:binomidb} &\bbinom{a}{t}_{q,\pi}=\begin{cases}\frac{[a]_{q,\pi}^!}{[t]_{q,\pi}^![a-t]_{q,\pi}^!}&\text{if } 0\leq t\leq a\\
0&\text{if } a<t\end{cases}\quad \text{if } a\geq 0, \\
\label{eq:binomidc} &\prod_{j=0}^{a-1}\parens{1+(\pi q^{2})^jz}
=\sum_{t=0}^a\pi^{\binom{t}{2}}q^{t(a-1)}\bbinom{a}{t}_{q,\pi} z^t
 \quad \text{if } a\geq 0.
\end{align}
Here $z$ is another indeterminate.
If $a', a''$ are integers and
$t\in\N$, then
\begin{equation}\label{eq:binomide}
\bbinom{a'+a''}{t}_{q,\pi}=\sum_{t'+t''=t}
\pi^{t't''+a't''}
q^{a't''-a''t'}
\bbinom{a'}{t'}_{q,\pi}\bbinom{a''}{t''}_{q,\pi}.
\end{equation}
We note the following specializations of some of the above identities.
Observe that $$\bbinom{-1}{t}_{q,\pi}=(-1)^t\pi^{\binom{t+1}{2}}$$ for any
$t\geq 0$, $i\in I$.
Furthermore if $a\geq 1$, then we have
\begin{equation}\label{eq:binomidf}
\sum_{t=0}^a(-1)^t\pi^{\binom{t}{2}}q^{t(a-1)}\bbinom{a}{t}_{q,\pi}=0
\end{equation}
which follows from \eqref{eq:binomidc} by setting $z=-1$.

We will use the notation
$$
q_i=q^{d_i}, \quad \pi_i=\pi^{d_i}, \quad \text{ for } i\in I.
$$
More generally, for $\nu=\sum \nu_i \alpha_i$, we set
\[q_\nu=\prod_{i\in I} q_i^{\nu_i},\quad \pi_\nu=\prod_{i\in I} \pi_i^{\nu_i}.\]
We also extend this notation to quantum integers, factorials, and binomial coefficients;
that is, we set
\[[n]_i=[n]_{q_i,\pi_i},\qquad [n]_i^!=\bra{n}_{q_i,\pi_i}^!,\qquad \bbinom{n}{k}_{i}=\bbinom{n}{k}_{q_i,\pi_i}.\]

The {\em bar involution} on $\Qqp$
is the $\Q^\pi$-algebra automorphism defined by $\bar{f(q)}=f(\pi q^{-1})$ for $f(q)\in \Qqp$.
We note that the bar involution restricts to a $\Zp$-algebra automorphism of $\A$
and that the $(q,\pi)$-integers are bar-invariant.

\subsection{The quantum covering groups}\label{subsec:cqg}

We recall some definitions from \cite{CHW1}.

\begin{definition}\cite{CHW1}\label{def:hcqg}
The half-quantum covering group $\ff$ associated to the anisotropic datum
$(I,\cdot)$ is the $Q^+$-graded $\Qqp$-algebra
on the generators $\theta_i$ for $i\in I$ with $|\theta_i|=\alpha_i$,
satisfying the relations
\begin{equation}\label{eq:thetaserrerel}
\sum_{k=0}^{b_{ij}} (-1)^k\pi^{\binom{k}{2}p(i)+kp(i)p(j)}
\bbinom{b_{ij}}{k}_{i}  \theta_i^{b_{ij}-k}\theta_j\theta_i^k=0
\;\; (i\neq j),
\end{equation}
where here recall that $b_{ij}=1-a_{ij}$ (cf. \S 2.1).
\end{definition}

We define the divided powers
\[\theta_i^{(n)}=\theta_i^{n}/\bra{n}_{i}^!.\]
Let $\fint$ be the $\A$-algebra generated by $\theta_i^{(n)}$ for various $i\in I$, $n\in \N$.

The algebra $\ff$ admits a coproduct structure. To wit, we equip $\ff\otimes\ff$
with the twisted multiplication
\begin{equation}\label{eq:twistedmul}
(x\otimes y)(x'\otimes y')=\pi^{p(x')p(y)}q^{-(|x'|,|y|)}(xx')\otimes (yy'),
\end{equation}
and obtain a $\Qqp$-algebra homomorphism $r:\ff\rightarrow \ff\otimes \ff$ satisfying $r(\theta_i)=\theta_i\otimes 1+1\otimes \theta_i$.
We note that this map satisfies
\begin{equation}\label{eq:r on divpow}
r(\theta_i^{(n)})=\sum_{s+t=n} (\pi q)^{-st} \theta_i^{(s)}\otimes \theta_i^{(t)}
\end{equation}
There are unique $\Qqp$-linear maps
$\ri,\ir:\ff\rightarrow\ff$ for each $i\in I$ such that $\ri(1)=\ir(1)=0$ and
$\ri(\theta_j)=\ir(\theta_j)=\delta_{ij}$ and satisfying
\begin{equation}\label{eq:twderiv}
\begin{aligned}
\ir(xy)&=\ir(x)y+\pi^{p(x)p(i)}q^{-(|x|,\alpha_i)}x\ir(y),\\
\ri(xy)&=\pi^{p(y)p(i)}q^{-(|y|,\alpha_i)}\ri(x)y+x\ri(y).
\end{aligned}
\end{equation}
Moreover, $r(x)=\ri(x)\otimes \theta_i+ \theta_i\otimes \ir(x)+(\text{other bi-homogeneous terms})$.

Finally, we recall that $\ff$ comes equipped with a symmetric bilinear form $(-,-)$ satisfying
\begin{equation}\label{eq:bilform}
\begin{aligned}
(1,1)&=1;\\
(\theta_i, \theta_j) &= \delta_{ij} (1-\pi_i q_i^{2})^{-1} \quad
(\forall i,j\in I);\\
(x, y'y'') &= (r(x), y'\otimes y'') \quad (\forall x,y',y'' \in
\ff.
\end{aligned}
\end{equation}
Here, the induced bilinear form $\ff \otimes \ff$ on $\ff$ is given by
\begin{equation}
(x_1\otimes x_2, x_1' \otimes x_2') :=
(x_1,x_1')(x_2, x_2'),
\end{equation}
for homogeneous $x_1,x_2, x_1',x_2' \in\ff$.
In particular, for all $x,y\in \ff$ we have
\begin{equation}\label{eq:derivsInnerProd}
(\theta_ix,y)=(\theta_i,\theta_i)(x,\ir(y)),\qquad (x\theta_i,y)=(\theta_i,\theta_i)(x,\ri(y)).
\end{equation}
\begin{definition} \cite{CHW1}
 \label{dfn:cqg}
The quantum covering group
$\UU$ associated to the datum $(P,P^\vee,\Pi,\Pi^\vee)$ is the $\Qqp$-algebra with generators
$E_i, F_i$, $K_\mu$, and $J_\mu$, for  $i\in I$ and $\mu\in P^\vee$, subject to the
relations:
\begin{equation}\label{eq:JKrels}
J_\mu J_\nu=J_{\mu+\nu},\quad K_\mu K_\nu=K_{\mu+\nu},\quad K_0=J_0=J_\nu^2=1,\quad
J_\mu K_\nu=K_\nu J_\mu,
\end{equation}
\begin{equation}\label{eq:Jweightrels}
J_\mu E_i=\pi^{\ang{\mu,\af_i}} E_i J_\mu,\quad J_\mu F_i=\pi^{-\ang{\mu,\af_i}} F_i J_\mu,
\end{equation}
\begin{equation}\label{eq:Kweightrels}
K_\mu E_i=q^{\ang{\mu,\af_i}} E_i K_\mu,\quad K_\mu F_i=q^{-\ang{\mu,\af_i}} F_i K_\mu,
\end{equation}
\begin{equation}\label{eq:commutatorrelation}
E_iF_j-\pi^{p(i)p(j)}F_jE_i=\delta_{ij}\frac{J_{d_i i}K_{d_i \alpha_i^\vee}-K_{-d_i \alpha_i^\vee}}{\pi_i q_i- q_i^{-1}},
\end{equation}
\begin{equation}\label{eq:Eserrerel}
\sum_{k=0}^{b_{ij}} (-1)^k\pi^{\binom{k}{2}p(i)+kp(i)p(j)}\bbinom{b_{ij}}{k}_{i}
 E_i^{b_{ij}-k}E_jE_i^k=0 \;\; (i\neq j),
\end{equation}
\begin{equation}\label{eq:Fserrerel}
\sum_{k=0}^{b_{ij}} (-1)^k\pi^{\binom{k}{2}p(i)+kp(i)p(j)}\bbinom{b_{ij}}{k}_{i}
 F_i^{b_{ij}-k}F_jF_i^k=0 \;\; (i\neq j),
\end{equation}
for $i,j\in I$ and $\mu,\nu\in P^\vee$.
\end{definition}

We endow $\UU$ with a $Q_+$-grading by setting
\begin{equation}
|E_i|=\af_i,\quad |F_i|=-\af_i,\quad |J_\mu|=|K_\mu|=0,
\end{equation}
and also endow $\UU$ with a $\Z_2$-grading by setting
\begin{equation}
p(E_i)=p(F_i)=p(i),\quad p(J_\mu)=p(K_\mu)=0.
\end{equation}
We set $\UU_\nu=\set{x\in \UU: |x|=\nu}$.
Note that $p(x)=p(\nu)$ for all $x\in \UU_\nu$.
Henceforth, any equation involving $|-|$ or $p(-)$ implicitly assumes
all the elements are homogeneous.

Let $\Up$ be the subalgebra generated by $E_i$ with $i\in I$,
and $\Uz$ be the subalgebra generated by $K_\nu$ and $J_\nu$ for
$\nu\in Y$. There is an isomorphisms $\ff\rightarrow\Um$
(resp. $\ff\rightarrow \Up$) defined by $\theta_i\mapsto \theta_i^-=F_i$
(resp. $\theta_i\mapsto \theta_i^+=E_i$). The following proposition was proven
in \cite{CHW1}.

\begin{proposition}\label{prp:UUtriang}
There is a triangular decomposition
\[\UU\cong \Um\otimes\Uz\otimes\Up\cong \Up\otimes \Uz\otimes \Um.\]
\end{proposition}

We define the divided powers
\[E_i^{(n)}=(\theta_i^{(n)})^+,\quad F_i^{(n)}=(\theta_i^{(n)})^-,\]
and set $\Uint^\pm=(\fint)^\pm$.
We will also use the shorthand notations
\[\tJ_i=J_{d_i\alpha_i^\vee},\quad \tJ_\nu=J_{\tilde \nu},\quad \tK_{i}=K_{d_i\alpha_i^\vee},\quad \tK_{\nu}=K_{\tilde \nu}.\]
Then for $\nu\in P^\vee$, we also have the $\nu$-integers and $\nu$-binomial coefficients
\[[\nu;n]=\frac{\pi_\nu^n v_\nu^n\tJ_\nu \tK_\nu-\tK_\nu^{-1}v_\nu^{-n}}{\pi_\nu v_\nu-v_\nu^{-1}},\quad \bbinom{\nu; n}{k}=\frac{\prod_{s=1}^k [\nu; n+1-k]}{[k]_{v_\nu,\pi_\nu}^!}.\]
We let $\Uint$ be the $\A$-subalgebra of $\UU$ generated by $E_i^{(n)}$, $F_i^{(n)}$, $J_\nu$, and $K_\nu$ for $i\in I$, $\nu\in Y$, $n\geq a\in \N$.

We have the following general commutation lemma. (See \cite[Proposition 2.2.2]{CHW1}.)

\begin{proposition}\label{prop:EFcommutation}
For $x\in \ff$ and $i\in I$, we have (in $\UU$)
\vspace{.1in}
\begin{enumerate}
\item[(a)] $\displaystyle
[x^+,F_i]=\frac{\pi_i^{p(x)-p(i)}\tJ_i\tK_i\ \ir(x)^+-\ri(x)^+\tK_{-i}\,\,
\, }{\pi_iq_i-q_i^{-1}},
$
\vspace{.1in}
\item[(b)] $\displaystyle
[E_i,x^-]=\frac{\pi_i^{p(x)-p(i)}\ri(x)^-\tJ_i\tK_{i}
 -\tK_{-i}\,\,\ir(x)^-}{\pi_iq_i-q_i^{-1}}.
$
\end{enumerate}
\end{proposition}
Specializing this identity yields the following relation in $\Uint$.

\begin{lemma}\label{L:commutation}\cite[Lemma 2.8]{CW} For $i\in I$, and $N,M\geq 1$,
\begin{enumerate}
\item[(a)] $\displaystyle E_i^{(N)}F_i^{(M)}=\sum_{t\geq0}\pi_i^{MN-{{t+1}\choose 2}}F_i^{(M-t)}\left[\alpha_i^\vee;2t-N-M\atop t\right]E_i^{(N-t)},$
\item[(b)] $\displaystyle F_i^{(N)}E_i^{(M)}=\sum_{t\geq0}(-1)^t\pi_i^{MN-t(M+N)}E_i^{(M-t)}\left[\alpha_i^\vee;M+N-t-1 \atop t\right]F_i^{(N-t)},$
\end{enumerate}
where we interpret $F_i^{(0)}=E_i^{(0)}=1$, and $F_i^{(s)}=E_i^{(s)}=0$ if $s<0$.
\end{lemma}

The algebra $\UU$ has a number of important automorphisms, which we will now recall.
There is a $\Qqp$-algebra automorphism $\omega:\UU\rightarrow\UU$
defined by
\begin{equation}\label{eq:omegadef}
\omega(E_i)=\pi_i\tJ_iF_i,\quad
\omega(F_i)=E_i,\quad
\omega(K_\nu)=K_{-\nu},\quad
\omega(J_\nu)=J_\nu.
\end{equation}

There is also an important anti-automorphism of $\UU$. To wit, there is a $\Qqp$-linear map $\sm:\UU\rightarrow\UU$
such that
\begin{equation}\label{eq:rhodef}
\sigma(E_i)=E_i,\quad
\sigma(F_i)=\pi_i\tJ_iF_i,\quad
\sigma(K_\nu)=K_{-\nu},\quad
\sigma(J_\nu)=J_\nu,
\end{equation}
and satisfying
\[\sm(xy)=\sm(y)\sm(x).\]

The bar-involution on $\UU$ is the $\Q^\pi$-algebra automorphism defined by
\begin{equation}\label{eq:bardef}
\bar E_i=E_i,\quad
\bar F_i=F_i,\quad
\bar K_\nu=J_\nu K_{-\nu},\quad
\bar J_\nu=J_\nu,\quad
\bar q= \pi q^{-1}.
\end{equation}
The maps $\omega$, $\sm$, and $\bar{\phantom{x}}$ (or variations thereof) were defined in \cite{CHW1}.

Finally, we recall that $\UU$ has a braided Hopf algebra structure. Specifically, endowing $\UU\otimes \UU$
with the multiplication $(x\otimes y)(x'\otimes y')=\pi^{p(x')p(y)}(xx')\otimes(yy')$, the map
$\Delta:\UU\rightarrow \UU\otimes \UU$ satisfying
\begin{align*}
\Delta(E_i) &= E_i\otimes 1 + \tJ_i\tK_i\otimes E_i\quad
(i\in I)
 \\
\Delta(F_i) &= F_i\otimes \tK_{i}^{-1} + 1\otimes F_i\quad (i\in
I)
 \\
\Delta(K_\mu) &=K_\mu\otimes K_\mu\quad (\mu\in Y)
 \\
\Delta(J_{\mu}) &=J_{\mu}\otimes J_{\mu}\quad (\mu\in Y).
\end{align*}
is an algebra homomorphism. This is related to the coproduct $r$ on $\ff$ as follows. Given $x\in \ff$ such that $r(x)=\sum x_1\otimes x_2$, then
\begin{equation}\label{eq:coprod f vs U}
\begin{aligned}
\Delta(x^+)&=\sum \pi^{p(x_1)p(x_2)}q^{(|x_1|,|x_2|)}x_2^+\tJ_{|x_2|}\tK_{|x_2|}\otimes x_1^+\\
\Delta(x^-)&=\sum x_1^-\otimes \tK_{-|x_1|}x_2^-
\end{aligned}
\end{equation}
Moreover, we have the formulas
\begin{equation}\label{eq:coprodDivPow}
\begin{aligned}
\Delta(E_i^{(p)})
 &=\sum_{p'+p''=p}q_i^{p'p''}\tJ_i^{p''}E_i^{(p')}\tK_i^{p''}\otimes E_i^{(p'')}, \\
\Delta(F_i^{(p)})
 &=\sum_{p'+p''=p}(\pi_i q_i)^{-p'p''}F_i^{(p')}\otimes\tK_i^{-p'}
 F_i^{(p'')}.
\end{aligned}
\end{equation}

\subsection{Representation categories}

In this paper, a $\UU$-module will always mean a $\Qqp$-module which carries a $\UU$-action
and a $\Z/2\Z$-grading compatible with the action.
Recall that a weight module for $\UU$ is a $\UU$-module $M$ such that
\[M=\bigoplus_{\lambda\in P} M_\lambda,\quad M_\lambda=\set{m\in M\mid K_\mu m=q^{\ang{\mu,\lambda}} m\text{ for any }\mu\in P^\vee}.\]
We say that a weight module $M$ is $\pi$-free if $M_\lambda$ is free as a $\Qqp$-module.
Henceforth, we shall always assume a $\UU$-module is a $\pi$-free weight module.

An important subcategory of $\UU$-modules is the category $\catO$ of $\pi$-free weight modules $M$ such that for any $m\in M$,
there exists an $N$ such that $x^+m=0$ for any $x\in \ff$ with $\height|x|>N$. The category $\catO$ in turn has an
important subcategory $\catO_{\rm int}$ formed by its integrable modules; that is, modules $M\in \catO$ such that $E_i$ and $F_i$
act locally nilpotently for all $i\in I$. We recall from \cite[\S2.6]{CHW1} that $\catO_{\rm int}$ is completely reducible,
with simple modules $V(\lambda)$ for $\lambda\in P_+$. Moreover, these modules arise as quotients
of standard highest weight modules $M(\lambda)$ (each of which is  isomorphic to $\ff$ as a vector space).

When studying the braid group action, it is often sufficient to restrict attention to a particular simple root. To that end,
let $\UU(i)$ be the subalgebra of $\UU$ generated by $E_i$, $F_i$, $\tK_{i}$ and $\tJ_{i}$.
We define the notation $\catO^i$ (respectively, $\catO^i_{\mathrm{int}}$)  for the corresponding categories
of $\UU(i)$-modules. Then the weights of $\UU(i)$ may, and shall, be identified with integers $\Z$ (see \cite{CW}).

From Lemma \ref{L:commutation} we have the following immediate corollary.

\begin{corollary}\label{C:commutation} Let $M\in\catO^i_{\mathrm{int}}$, and let $m\in \Z_{\geq0}$. Assume $\eta\in M_m$ satisfies $E_i\eta=0$, and let $\xi=F_i^{(m)}\eta$. Then, for $k,h\geq0$ such that $k+h=m$,
$$F_i^{(k)}\eta=\pi_i^{mh+{h+1\choose2}}E_i^{(h)}\xi.$$
\end{corollary}

We note the following lemma.

\begin{lemma}\label{L:ModuleMaps} Let $M\in\catO^i_{\mathrm{int}}$ be an irreducible $\UU(i)$ module of highest weight $m\in\Z_{\geq0}$. Let $\eta\in M_m$ satisfy $E_i\eta=0$ and let $\xi=F_i^{(m)}\eta$.
\begin{enumerate}
\item[(a)] There is a $\Q(q)^\pi$-linear map $\om:M\longrightarrow M$ defined by $\om(\eta)=\pi_i^{m\choose 2}\xi$, $\om(\xi)=\eta$, and $\om(u.\eta)=\om(u).\om(\eta)$ for all $u\in \UU(i)$. Moreover, $\om^4=1$.
\item[(b)] There is a $\Q$-linear involution $\bar{\phantom{x}}:M\longrightarrow M$ defined by $\overline{q}=\pi q^{-1}$, $\overline{\pi}=\pi$, $\overline{\eta}=\eta$ and $\overline{u.\eta}=\overline{u}.\overline{\eta}$ for all $u\in \UU(i)$.
\end{enumerate}
\end{lemma}
Using the semisimplicity of the category $\catO^i_{\mathrm{int}}$, we obtain the following corollary.
\begin{corollary}\label{C:ModuleMaps}
Let $M\in\catO^i_{\mathrm{int}}$.
\begin{enumerate}
\item[(a)] There is a $\Q(q)^\pi$-linear map $\om:M\longrightarrow M$ such that $\om(u\eta)=\om(u)\om(\eta)$ for all $u\in \UU(i)$ and $\eta\in M$. Moreover, $\om^4=1$.
\item[(b)] There is a $\Q$-linear involution $\bar{\phantom{x}}:M\longrightarrow M$ defined by $\overline{q}=\pi q^{-1}$, $\overline{\pi}=\pi$, and $\overline{u.\eta}=\overline{u}.\overline{\eta}$ for all $u\in \UU(i)$ and $\eta\in M$.
\end{enumerate}
\end{corollary}
Note that there are many possible choices of such maps for an arbitrary $M\in \catO^i_{\rm int}$,
but we shall not need a particular choice.

\subsection{Higher Serre Relations}

The higher Serre relations were examined in detail in \cite[\S4]{CHW1}, and we will recall the essential definitions and results.
To begin, for $i,j\in I$, and $n,m\geq 0$, set
\begin{align}
p(n,m;i,j)=mnp(i)p(j)+{m\choose 2}p(i)
\end{align}
and, for $i\neq j$, define the elements
\begin{align}
\label{E:eijnm}e_{i,j;n,m}&=\sum_{r+s=m}(-1)^r\pi_i^{p(n,r;i,j)}(\pi_iq_i)^{-r(na_{ij}+m-1)}E_i^{(r)}E_j^{(n)}E_i^{(s)},\\
\label{E:e'ijnm}e'_{i,j;n,m}&=\sum_{r+s=m}(-1)^r\pi_i^{p(n,r;i,j)}q_i^{-r(na_{ij}+m-1)}E_i^{(s)}E_j^{(n)}E_i^{(r)},\\
\label{E:fijnm}f_{i,j;n,m}&=\sum_{r+s=m}(-1)^r\pi_i^{p(n,r;i,j)}(\pi_iq_i)^{r(na_{ij}+m-1)}F_i^{(s)}F_j^{(n)}F_i^{(r)},\\
\label{E:f'ijnm}f'_{i,j;n,m}&=\sum_{r+s=m}(-1)^r\pi_i^{p(n,r;i,j)}q_i^{r(na_{ij}+m-1)}F_i^{(r)}F_j^{(n)}F_i^{(s)}.
\end{align}
When there is no chance of confusion, we will abbreviate $e_{i,j;n,m}=e_{n,m}$, etc. Note that we have the equalities
$$e'_{n,m}=\sm(e_{n,m}),\;\;\;f'_{n,m}=\sm(f_{n,m}),\;\;\;e_{n,m}=\om(\overline{f_{n,m}}),\andeqn e'_{n,m}=\om(\overline{f'_{n,m}}).$$

The following results were proved in \cite[\S 4]{CHW1}.

\begin{lemma}\label{L:HigherSerre} The following statements hold:
\begin{enumerate}
\item[(a)] \begin{align*}
\displaystyle E_i^{(N)}e_{n,m}=\sum_{k=0}^N&(-1)^kq_i^{N(na_{ij}+2m)+(N-1)k}\pi_i^{N(np(j)+m)+{k\choose 2}}\\
&\times\left[{m+k\atop k}\right]_ie_{n,m+k}E_i^{(N-k)};
\end{align*}
\item[(b)]  \begin{align*}
\displaystyle F_i^{(M)}e_{n,m}=\sum_{h=0}^M&(-1)^hq_i^{-(M-1)h}\pi_i^{M(m+np(j))+(M-m)h}\\
&\times \left[{-na_{ij}-m+h\atop h}\right]_iK_i^{-h}e_{n,m-h}F_i^{(M-h)}.
\end{align*}
\item[(c)] If $m>-na_{ij}$, then $e_{i,j;n,m}=0$.
\end{enumerate}
\end{lemma}

\section{Braid group operators}

We shall now define certain operators on $\UU$ and its integrable modules.
These operators are generalizations of Lusztig's braid operators on quantum groups; see \cite{Lu}.
Many of our results are direct generalizations of Lusztig's results in {\em loc. cit.} to the
quantum covering group setting.

\subsection{The symmetries $T_{i}$ and $T_{i}^{-1}$ of category $\catO$}
Fix $i\in I$. Let $M\in\catO^i_{\mathrm{int}}$. We define the $\Q(q)^\pi$-linear maps $T_{i}',T_{i}'':M\longrightarrow M$ by
\begin{equation}\label{eq:Tidef}
\begin{aligned}
T_{i}'(z)&=\sum_{\substack{a,b,c\geq0\\a-b+c=n}}(-1)^b\pi_i^{c}q_i^{-ac+b}\tJ_i^cF_i^{(a)}E_i^{(b)}F_i^{(c)}z;\\
T_{i}''(z)&=\sum_{\substack{a,b,c\geq0\\-a+b-c=n}}(-1)^{b}\pi_i^{ac+c+{n\choose2}}q_i^{ac-b}\tJ_i^aE_i^{(a)}F_i^{(b)}E_i^{(c)}z,
\end{aligned}
\end{equation}
when $z\in M_n$.
We observe that
\begin{equation}\label{eq:braidparity}
p(T_i'(z))=p(T_i''(z))=p(z)+np(i).
\end{equation}

\begin{remark} Let $M\in \catO^i$.
For $X\in \UU(i)$, define the formal power series
$$\exp(X)=\sum_{t} q_i^{-\binom{t}{2}}\frac{X^t}{[t]^!}.$$
Then $\exp(X)$ defines an operator on any module for which the action of $X$ is locally nilpotent.
Further define $q^{\binom{\alpha_i^\vee}{2}}:M\rightarrow M$ via
$$q^{\binom{\alpha_i^\vee}{2}}(m)=q_i^{\binom{n}{2}}m\;\;\;\mbox{for }m\in M_n.$$
It can be shown that $T_i'=\exp(q_i^{-1}F_i\tK_{i})\exp(-E_i)\exp(\pi_iq_iF_i\tJ_i\tK_{i})q^{\binom{\alpha_i^\vee}{2}}$, cf. \cite{Sai}.
\end{remark}

We can relate the maps $T_i'$ and $T_i''$ using the module automorphisms from Lemma \ref{L:ModuleMaps}.

\begin{lemma}\label{L:OmegaT}
Let $M\in \catO^i_{\rm int}$. Then for $z\in M_n$,
\begin{enumerate}
\item[(a)] $\om^2(T_i'(\om^2(z)))=T_i'(z)$,
\item[(b)] $T_i''(z)=\pi_i^{n+1\choose 2}\overline{\om\left(T_i'(\om^{-1}(\overline{z}))\right)}=\pi_i^{n+1\choose 2}\overline{\om^{-1}\left(T_i'(\om(\overline{z}))\right)}$,
\item[(c)] $T_i'(z)=\pi_i^{n+1\choose 2}\overline{\om\left(T_i''(\om^{-1}(\overline{z}))\right)}=\pi_i^{n+1\choose 2}\overline{\om^{-1}\left(T_i''(\om(\overline{z}))\right)}.$
\end{enumerate}
\end{lemma}

\begin{proof} Assume $z\in M_n$. Then $\overline{z}\in M_n$, so
\begin{align*}
\overline{\om\left(T_i'(\om^{-1}(\overline{z}))\right)}
    &=\overline{\om\left(\sum_{\substack{a,b,c\geq0\\a-b+c=n}}(-1)^b\pi_i^{c}q_i^{-ac+b}\tJ_i^cF_i^{(a)}E_i^{(b)}F_i^{(c)}\om^{-1}(\overline{z})\right)}\\
    &=\overline{\sum_{\substack{a,b,c\geq0\\a-b+c=n}}(-1)^b\pi_i^{c}q_i^{-ac+b}J_i^cE_i^{(a)}\pi_i^b\tJ_i^bF_i^{(b)}E_i^{(c)}\overline{z}}\\
    &=\sum_{\substack{a,b,c\geq0\\a-b+c=n}}(-1)^b\pi_i^{ac+c}q_i^{-ac+b}\tJ_i^{b+c}E_i^{(a)}F_i^{(b)}E_i^{(c)}z\\
    &=\pi_i^{n+1\choose 2}T_i''(z).
\end{align*}
In the last line, we have used the fact that $\tJ_i^{b+c}|_{M_n}=\pi_i^n \tJ_i^a|_{M_n}$. This proves the first equality in (b).

Next, using the definition of $\om$, we compute
\begin{align*}
\om^2(T_i'(\om^2(z)))&=\om^2\left(\sum_{\substack{a,b,c\geq0\\a-b+c=n}}(-1)^b\pi_i^{c}q_i^{-ac+b}\tJ_i^cF_i^{(a)}E_i^{(b)}F_i^{(c)}\om^2(z)\right);\\
    &=\sum_{\substack{a,b,c\geq0\\a-b+c=n}}(-1)^b\pi_i^{c}q_i^{-ac+b}\tJ_i^c(\pi_i^a\tJ_i^aF_i^{(a)})(\pi_i^b\tJ_i^bE_i^{(b)})(\pi_i^c\tJ_i^cF_i^{(c)})z\\
    &=\sum_{\substack{a,b,c\geq0\\a-b+c=n}}(-1)^b\pi_i^{c}q_i^{-ac+b}\tJ_i^cF_i^{(a)}E_i^{(b)}F_i^{(c)}(\pi_i^n\tJ_i^n)z.
\end{align*}
Part (a) follows since $\pi_i^n\tJ_i^nz=\pi_i^n(\pi_i^n)^nz=\pi_i^{n(n+1)}z=z$.

The second equality in (b) now follows. Finally, (c) follows from (a) and (b) since $\om$ commutes with the bar involution.
\end{proof}
The symmetries $T_i'$ and $T_i''$ can be computed explicitly on each simple module of $\catO^i$. In particular,
we have the following lemma.

\begin{lemma}\label{L:BraidGroupOnModules}
Let $M\in\catO^i$, and $m\in\Z_{\geq0}$. For $k,h\geq0$ such that $k+h=m$,
\begin{enumerate}
\item[(a)] If $\eta\in M_m$ satisfies $E_i\eta=0$, then
$$T_{i}'(F_i^{(k)}\eta)=(-1)^k\pi_i^{mk+{{k+1}\choose 2}}q_i^{hk+k}F_i^{(h)}\eta;$$
\item[(b)] If $\xi\in M_{-m}$ satisfies $F_i\xi=0$, then
$$T_{i}''(E_i^{(k)}\xi)=(-1)^k\pi_i^{mh+{h+1\choose2}}q_i^{-hk-k}E_i^{(h)}\xi.$$
\end{enumerate}
\end{lemma}

\begin{proof}

First note that (b) follows from (a). Indeed, observe that $E_i^{(k)}\xi\in M_{k-h}$, so by Lemmas \ref{L:ModuleMaps} and \ref{L:OmegaT},
\begin{align*}
T_i''(E_i^{(k)}\xi)&=\pi_i^{k-h+1\choose 2}\overline{\om\left(T_i'(\om^{-1}(\overline{E_i^{(k)}\xi}))\right)}\\
    &=\pi_i^{{k-h+1\choose 2}+{k-h\choose 2}}\overline{\om\left(T_i'(F_i^{(k)}\overline{\eta})\right)}\\
    &=\pi_i^{{k-h+1\choose 2}+{k-h\choose 2}}\overline{\om\left((-1)^k\pi_i^{mk+{k+1\choose2}}q_i^{hk+k}F_i^{(h)}\overline{\eta}\right)}\\
    &=(-1)^k\pi_i^{{k-h+1\choose 2}+mk+{k+1\choose 2}}(\pi_iq_i)^{-hk-k}E_i^{(h)}\xi.
\end{align*}
Then part (b) now follows from part (a) and the congruence
$${k-h+1\choose2}+mk+{k+1\choose2}+hk+k\equiv mh+{h+1\choose 2}\;\;\;(\mbox{mod }2).$$

It remains to prove (a). Assume $a-b+c=m-2k$. Using Lemma \ref{L:commutation}, we have
\begin{align*}
F_i^{(a)}E_i^{(b)}F_i^{(c)}&F_i^{(k)}\eta=\left[c+k\atop c\right]_iF_i^{(a)}E_i^{(b)}F_i^{(c+k)}\eta\\
    &=\sum_{t\geq0}\left[c+k\atop c\right]_i\left[b-c+h\atop t\right]_i\pi_i^{b(c+k)+{{t+1}\choose2}}F_i^{(a)}F_i^{(c+k-t)}E_i^{(b-t)}\eta.
\end{align*}
By assumption, $E_i^{(b-t)}\eta\neq 0$ only when $b=t$. Hence,
\begin{align*}
F_i^{(a)}E_i^{(b)}F_i^{(c)}F_i^{(k)}\eta
    &=\left[c+k\atop c\right]_i\left[b-c+h\atop b\right]_i\pi_i^{b(c+k)+{{b+1}\choose2}}F_i^{(a)}F_i^{(c+k-b)}\eta\\
    &=\left[c+k\atop c\right]_i\left[a+k\atop b\right]_i\left[h\atop a\right]_i\pi_i^{b(c+k)+{{b+1}\choose2}}F_i^{(h)}\eta,
\end{align*}
where we have used $a-b+c=h-k$ to make the substitution $b-c+h=a+k$ in the last line.
Since $\tJ_i^c$ acts on $F_i^{(k)}\eta$ as multiplication by $\pi_i^{c(h-k)}$, we see that it suffices
to show that
\begin{equation*}\begin{aligned}
(-1)^k\pi_i^{mk+{k+1\choose2}}q_i^{hk+k}=\sum_{\substack{a,b,c\geq0\\a-b+c=h-k}}&(-1)^b\pi_i^{b(c+k)+{{b+1}\choose2}+c+c(h-k)}\!q_i^{-ac+b}\\
    &\times\bbinom{c+k}{c}_i\bbinom{a+k}{b}_i\!\bbinom{h}{a}_i.
\end{aligned}    
\tag{$\star$}
\end{equation*}

The equality ($\star$) can be proven directly by an argument similar 
to the $\pi=1$ specialization of ($\star$) given in the proof of
\cite[Proposition 5.2.21]{Lu} using \eqref{eq:binomide}.
Alternatively, ($\star$) can be deduced from the $\pi=1$ case by rewriting
the identity in $\pi q^2$; see the proof of \cite[Lemma 7.2]{CHW2} for a similar deduction.
\end{proof}

In particular, we arrive at the following relation between $T_i'$ and $T_i''$ as maps on modules in $\catO^i$.

\begin{proposition}\label{P:inversebraids}
We have $T_{i}'T_{i}''=T_{i}''T_{i}'=1:M_n\longrightarrow M_n$.
\end{proposition}

\begin{proof}
Let $m=h+k$, and $\eta$ be as in Lemma \ref{L:BraidGroupOnModules}. Define $\xi=F_i^{(m)}\eta$ so that by Corollary \ref{C:commutation} $\pi_i^{(h-k)h+{h+1\choose 2}}F_i^{(k)}\eta=E_i^{(h)}\xi$, and $F_i^{(h)}\eta=\pi_i^{(h-k)k+{k+1\choose2}}E_i^{(k)}\xi$. Then, using Lemma~\ref{L:BraidGroupOnModules}, we have
\begin{align*}
T_{i}''T_{i}'(F_i^{(k)}\eta)&=T_{i}''\left((-1)^k\pi_i^{mk+{{k+1}\choose 2}}q_i^{hk+k}F_i^{(h)}\eta\right)\\
        &=(-1)^k\pi_i^{mk+{{k+1}\choose 2}}q_i^{hk+k}T_{i}''\left(\pi_i^{mk+{k+1\choose2}}E_i^{(k)}\xi\right)\\
        &=(-1)^kq_i^{hk+k}(-1)^k\pi_i^{mh+{{h+1}\choose2}}q_i^{-hk-k}E_i^{(h)}\xi\\
        &=\pi_i^{mh+{h+1\choose2}}\pi_i^{mh+{h+1\choose2}}F_i^{(k)}\eta\\
        &=F_i^{(k)}\eta.
\end{align*}
Now, $M$ is generated by vectors of the form $F_i^{(k)}\eta$ as above, so $T_{i}''T_{i}'=1$. The remaining identity $T_i'T_i''=1$ can be deduced in the same fashion.
\end{proof}

In light of this result, we shall henceforth use the following notations:
\begin{equation}T_i=T_{i}'\andeqn T_i^{-1}=T_i''.
\end{equation}

\begin{lemma}\label{L:BarT} For $z\in M_t$,
$$\overline{T_i(\overline{z})}=(-1)^t\pi_i^{t\choose2}q_i^{t}T_i^{-1}(z).$$

\end{lemma}

\begin{proof}
We may assume that $z=F_i^{(k)}\eta=\pi_i^{mh+{h+1\choose 2}}E_i^{(h)}\xi$, where $m=k+h,\eta,\xi$ are as in Lemma \ref{L:BraidGroupOnModules}. In this case, $h=k+t$ and we will use the fact that $m\equiv t$ (mod 2) throughout the proof. By Lemma \ref{L:BraidGroupOnModules},
\begin{align*}
T_i^{-1}(z)&=\pi_i^{mh+{h+1\choose2}}T_i^{-1}(E_i^{(h)}\xi)\\
    &=(-1)^h\pi_i^{mh+{h+1\choose2}+mk+{k+1\choose2}}q_i^{-kh-h}E_i^{(k)}\xi\\
    &=(-1)^h\pi_i^{mk+{t\choose2}}q_i^{-kh-h}E_i^{(k)}\xi
\end{align*}
where, in the last line, we have used
\begin{align*}
mh+{h+1\choose2}+mk+{k+1\choose 2}\equiv mk+{t\choose2}.
\end{align*}
On the other hand, by Lemma \ref{L:BraidGroupOnModules},
\begin{align*}
\overline{T_i(\overline{z})}&=\overline{T_i(F_i^{(k)}\overline{\eta})}=\overline{(-1)^k\pi_i^{mk+{k+1\choose2}}q_i^{hk+k}F_i^{(h)}\overline{\eta}}\\
    &=(-1)^k\pi_i^{mk+{k+1\choose 2}+hk+k}q_i^{-hk-k}F_i^{(h)}\eta=(-1)^k\pi_i^{mk}q_i^{-hk-k}E_i^{(k)}\xi.
\end{align*}
The result follows.
\end{proof}

\begin{lemma}\label{L:BraidsonEis}
For any $z\in M_t$,
\begin{enumerate}
\item[(a)] $T_i(F_iz)=-q_i^{t}E_iT_i(z)$;
\item[(b)] $T_i^{-1}(F_iz)=-\pi_i^{t+1}q_i^{-t+2}E_iT_i^{-1}(z)$;
\item[(c)] $T_i(E_iz)=-\pi_i^{t+1}q_i^{-t-2}F_iT_i(z)$;
\item[(d)] $T_i^{-1}(E_iz)=-q_i^{t}F_iT_i^{-1}(z)$;
\item[(e)] $T_i(z)\in M_{-t}$;
\item[(f)] $T_i^{-1}(z)\in M_{-t}$.
\end{enumerate}
\end{lemma}

\begin{proof}
Properties (e) and (f) are clear by the definitions of $T_i$ and $T_i^{-1}$.
We also note that (d) follows from (a) and (c) follows from (b) using Proposition \ref{P:inversebraids}.
As the proofs of (a) and (b) are entirely similar, we shall only prove (a).

To this end, assume that $z=F_i^{(k)}\eta=\pi_i^{mh+{h+1\choose 2}}E_i^{(h)}\xi$, where $m=k+h,\eta,\xi$ are as in Lemma \ref{L:BraidGroupOnModules}. In this case, $h=k+t$ and we will repeatedly use the fact that $m\equiv t$ (mod 2) throughout the proof. Note that if $k=m$, then both sides of both (a) and (b) are zero. Therefore, assume $k<m$ and $h>0$. Then, for (a),
\begin{align*}
T_i(F_iz)&=[k+1]_iT_i(F_i^{(k+1)}\eta)\\
            &=(-1)^{k+1}\pi_i^{m(k+1)+{k+2\choose2}}q_i^{h(k+1)}[k+1]_iF_i^{(h-1)}\eta\\
            &=(-1)^{k+1}\pi_i^{m(k+1)+{k+2\choose2}+m(k+1)+{k+2\choose2}}q_i^{h(k+1)}[k+1]_iE_i^{(k+1)}\xi\\
            &=(-1)^{k+1}q_i^{h(k+1)}[k+1]_iE_i^{(k+1)}\xi,
\end{align*}
while
\begin{align*}
E_iT_i(z)&=E_iT_i(F_i^{(k)}\eta)\\
            &=(-1)^k\pi_i^{mk+{k+1\choose 2}}q_i^{(h+1)k}E_iF_i^{(h)}\eta\\
            &=(-1)^k\pi_i^{mk+{k+1\choose 2}+mk+{k+1\choose2}}q_i^{(h+1)k}E_iE_i^{(k)}\xi\\
            &=(-1)^kq_i^{(h+1)k}[k+1]_iE_i^{(k+1)}\xi.
\end{align*}
Therefore, part (a) follows since $h=k+t$.
\end{proof}

\subsection{Braid operators on $\catO_{\mathrm{int}}$}

Now, assume that $M\in\catO_{\mathrm{int}}$. Then $M$ can be regarded as an object of $\catO^i_{\mathrm{int}}$ for each $i\in I$, and we obtain an action of the symmetries
$$T_i,T_i^{-1}:M\longrightarrow M.$$
We call these the {\em braid operators} of $\catO_{\mathrm{int}}$. We note that $T_i,T_i^{-1}$ are not homogeneous with respect
to the $\Z/2\Z$-grading on $M$; however, they are homogeneous on each weight space. To wit,
for $\lambda\in P$ and homogeneous $m\in M_\lambda$, we have that
\begin{equation}\label{eq:parityofbraidoperators}
p(T_i(m))=p(T_i^{-1}(m))\equiv p(m)+p(i)\ang{\af_i^\vee,\lambda}\quad\text{ (mod 2)}
\end{equation}

\begin{lemma}\label{L:BraidsandKs}
Let $M\in\catO_{\mathrm{int}}$, and let $z\in M$. Fix $\mu\in P^\vee$, and let $$\nu=s_{\af_i^\vee}(\mu)=\mu-\la\mu,\af_i\ra\af_i^\vee\in P^\vee.$$
Then,
\begin{enumerate}
\item[(a)] $T_i^{-1}(K_\nu z)=K_\mu T_i^{-1}(z)$ and $T_i^{-1}(J_\nu z)=J_\mu T_i^{-1}(z)$;
\item[(b)] $T_i(K_\nu z)=K_\mu T_i(z)$ and $T_i(J_\nu z)=J_\mu T_i(z)$;
\end{enumerate}
Moreover,
\begin{enumerate}
\item[(c)] $T_i(\tJ_\nu z)=\tJ_\nu T_i(z)$;
\item[(d)] $T_i^{-1}(\tJ_\nu z)=\tJ_\nu T_i^{-1}(z)$.
\end{enumerate}
\end{lemma}

\begin{proof} Parts (a) and (b) are proved exactly as in \cite[Proposition 5.2.6]{Lu}. The main point is that if $z$ is a weight vector, say $z\in M_\ld$, then $T_i(z),T_i^{-1}(z)\in M_{\ld-\la\af_i^\vee,\ld\ra\af_i}$ by Lemma \ref{L:BraidsonEis}(e),(f). It is left to observe that $K_\mu$ (respectively, $J_\mu$) acts on the $\ld-\la\af_i^\vee,\ld\ra\af_i$ weight space as multiplication by $q^\bigstar$
(respectively, $\pi^\bigstar$), where
$$\bigstar=\la\mu,\ld\ra-\la\af_i^\vee,\ld\ra\la\mu,\af_i\ra=\la\nu,\ld\ra.$$

Finally, we prove (c) and (d). For this, note that $\tJ_\mu=J_{\tilde{\mu}}$ by definition and $\ang{\tilde\mu,\alpha_i}\in 2\Z$, which
in turn implies that $\la\mu, \lambda\ra\equiv\la\nu,\lambda\ra$ (mod 2).
\end{proof}

\begin{corollary}
Let $M\in\catO_{\mathrm{int}}$.
The maps $T_i,T_i^{-1}:M\rightarrow M$ restrict to bijections between $M_{\ld}$ and $M_{\ld-\la\af_i^\vee,\ld\ra\af_i}$.
\end{corollary}

Recall the elements $e_{n,m;i,j}$, $e'_{n,m;i,j}$, $f_{n,m;i,j}$, and $f'_{n,m;i,j}$ defined by \eqref{E:eijnm}-\eqref{E:f'ijnm}.

\begin{lemma}\label{L:BraidsonEjs} Let $i,j\in I$ and assume that $i\neq j$. Let $M$ be any object in $\catO_{\mathrm{int}}$ and let $in M$. We have
\begin{enumerate}
\item[(a)] $T_i^{-1}(e_{i,j;n,-na_{ij}}z)=\pi_i^{{na_{ij}\choose 2}}\tJ_i^{np(j)}E_j^{(n)}T_i^{-1}(z)$;
\item[(b)] $T_i(e'_{i,j;n,-na_{ij}}z)=\pi_i^{na_{ij}\choose2}\tJ_i^{np(j)}E_j^{(n)}T_i(z)$;
\item[(c)] $T_i^{-1}(f_{i,j;n,-na_{ij}}z)=\tJ_i^{np(j)}F_j^{(n)}T_i^{-1}(z)$;
\item[(d)] $T_i(f'_{i,j;n,-na_{ij}}z)=\tJ_i^{np(j)}F_j^{(n)}T_i(z)$.
\end{enumerate}
\end{lemma}

\begin{proof} As before, we write $e_{n,m}=e_{i,j;n,m}$.
We may assume $z\in M_\ld$ for some $\ld\in P$. Then, $e_{n,-na_{ij}}z\in M_{\ld-na_{ij}\af_i+n\af_j}$. Let
$$p=\la\af_i^\vee,\ld\ra\andeqn p'=\la \af_i^\vee,\ld-na_{ij}\af_i+n\af_j\ra=p-na_{ij}.$$
Note that, since $\tJ_i^ae_{n,-na_{ij}}z=\pi_i^{ap'}e_{n,-na_{ij}}z$,
$$T_i^{-1}(e_{n,-na_{ij}}z)=\sum_{\substack{a,b,c\geq0\\-a+b-c=p'}}(-1)^b\pi_i^{ac+c+{p'\choose2}+ap'}q_i^{ac-b}E_i^{(a)}F_i^{(b)}E_i^{(c)}e_{n,-na_{ij}}z.$$
By Lemma \ref{L:HigherSerre}, we have that
$$E_i^{(c)}e_{n,-na_{ij}}=q_i^{-cna_{ij}}\pi_i^{cnp(j)}e_{n,-na_{ij}}E_i^{(c)}.$$
Therefore, using Lemma \ref{L:HigherSerre}(b), then (a), we deduce that
\begin{align*}
&E_i^{(a)}F_i^{(b)}E_i^{(c)}e_{n,-na_{ij}}z=q_i^{-cna_{ij}}\pi_i^{cnp(j)}E_i^{(a)}F_i^{(b)}e_{n,-na_{ij}}E_i^{(c)}z\\
    &=q_i^{-cna_{ij}}\pi_i^{cnp(j)}E_i^{(a)}\sum_{b'=0}^b(-1)^{b'}q_i^{-(b-1)b'}\pi_i^{bnp(j)+bb'}K_i^{-b'}e_{n,-na_{ij}-b'}F_i^{(b-b')}E_i^{(c)}z\\
    &=\sum_{b'=0}^{b}\sum_{a'=0}^a(-1)^{b'+a'}q_i^{\spadesuit_0}\pi_i^{\clubsuit_0}\!\! \left[{-na_{ij}-b'+a'\atop a'} \right]_i\!\!e_{n,-na_{ij}-b'+a'}E_i^{(a-a')}F_i^{(b-b')}E_i^{(c)}z
\end{align*}
where
\begin{align*}\spadesuit_0&=-cna_{ij}-(b-1)b'-b'(p'+2(c-b))+a(na_{ij}-2na_{ij}-2b')+(a-1)a'\\
    &=-(a+c)na_{ij}-2ab'+bb'-2cb'-b'p'+b'+aa'-a',
\end{align*}
and
$$\clubsuit_0\equiv (-a+b-c)np(j)+(a+b)b'+{a'\choose 2}\;\;\;(\mbox{mod }2).$$
Introduce the variables $a''=a-a'$ and $b''=b-b'$. Then, summing over $a,b,c\geq0$ such that $-a+b-c=p'$ and, using the relation $na_{ij}=a'+a''-b'-b''+c+p$, we obtain
\begin{align}\label{E:bigeqn2}
\notag T_i^{-1}(e_{n,-na_{ij}}z)=\!\!\!\sum_{\substack{a,b,c\geq0\\-a+b-c=p'}}&\sum_{\substack{a',a''\geq0\\a''+a'=a}}\sum_{\substack{b',b''\geq0\\b''+b'=b}}(-1)^{b''+a'} q_i^\spadesuit\pi_i^\clubsuit
    \left[{-a''+b''-c-p\atop a'}\right]_i\\
    &\times e_{n,-a''+b''-c-p}E_i^{(a'')}F_i^{(b'')}E_i^{(c)}z,
\end{align}
where
$$\spadesuit=ac-b+\spadesuit_0=a'(-1+b''-a''-c-p)+(a''c-b'')+(a''+c)(-a''+b''-c-p),$$
and
\begin{align*}
\clubsuit=&ac+c+{p'\choose 2}+ap'+\clubsuit_0\\
    \equiv&\left[{na_{ij}\choose 2}+pnp(j)\right]+(a''c+c+{p\choose 2}+a''p)\\&+(a'c+a'p+(a+b)b')+{a'\choose 2}.
\end{align*}
Using the congruence $a+b+c\equiv p$ (mod 2), we can rewrite
$$a'c+a'p+(a+b)b'\equiv (a+b)(a'+b')\equiv (c+p)(a''+b''+c+p).$$
Hence,
$$\clubsuit\equiv {na_{ij}\choose 2}+pnp(j)+a''c+c+{p\choose 2}+a''p+(c+p)(a''+b''+c+p)+{a'\choose 2}.$$
By Lemma \ref{L:HigherSerre} and the definitions, $e_{n,-a''+b''-c-p}=0$ unless $0\leq-a''+b''-c-p\leq -na_{ij}$. We may therefore add this
condition without changing the sum. Now, from the equation $-a+b-c=p-na_{ij}$ and the previous inequality, we deduce that
the sum involving $b'$ is redundant, as $b'\geq 0$ is determined by $a',a'',b'',c$:
$$b'=a'+a''-b''+c+p-na_{ij}\geq0.$$
Therefore, the sum \eqref{E:bigeqn2} becomes
\begin{align}\label{E:bigeqn3}
\notag\pi_i^{\clubsuit'''}\sum_{a'',b'',c''}(-1)^{b''}\pi_i^{\clubsuit''}q_i^{\spadesuit''}
   &\left( \sum_{a'\geq0}(-1)^{a'}\pi_i^{\clubsuit'}q_i^{\spadesuit'}\left[{-a''+b''-c-p\atop a'}\right]_i\right) \\ &\times e_{n,-a''+b''-c-p} E_i^{(a'')}F_i^{(b'')}E_i^{(c)}z,
\end{align}
where the first sum is over $a'',b'',c\geq0$ with $0\leq-a''+b''-c-p\leq-na_{ij}$,
$$\spadesuit'=a'(-1+b''-a''-c-p)\andeqn \spadesuit''=(a''c-b'')+(a''+c)(-a''+b''-c-p),$$
$$\clubsuit'={a'\choose 2},\;\;\; \clubsuit''=a''c+c+{p\choose 2}+a''p+(c+p)(a''+b''+c+p),$$
and
$$\clubsuit'''={na_{ij}\choose 2}+pnp(j).$$
Now, we deduce from \eqref{eq:binomidf} that the sum over $a'$ in \eqref{E:bigeqn3} is 0 unless $-a''+b''-c-p=0$. Summarizing the above computation, we have
\begin{align*}
T_i^{-1}(e_{n,-na_{ij}}z)&=\pi_i^{\clubsuit'''}e_{n,0}\hspace{-1em}\sum_{\substack{a'',b'',c''\geq0\\-a''+b''-c=p}}\hspace{-1em}
(-1)^{b''}\pi_i^{a''c+c+{p\choose 2}+a''p}q_i^{a''c-b''}E_i^{(a'')}F_i^{(b'')}E_i^{(c)}z\\
    &=\pi_i^{\clubsuit'''}e_{n,0}\hspace{-1em}\sum_{\substack{a'',b'',c''\geq0\\-a''+b''-c=p}}\hspace{-1em}
(-1)^{b''}\pi_i^{a''c+c+{p\choose 2}}q_i^{a''c-b''}\tJ_i^{a''}E_i^{(a'')}F_i^{(b'')}E_i^{(c)}z\\
    &=\pi_i^{{na_{ij}\choose 2}}\tJ_i^{np(j)}E_j^{(n)}T_i^{-1}(z).
\end{align*}
This proves (a).

To prove (b), observe that by Lemma \ref{L:BarT},
$$\overline{T_i^{-1}(\overline{z})}=(-1)^{\la\af_i^\vee,\ld\ra}\pi_i^{\la\af_i^\vee,\ld\ra\choose2}q_i^{\la\af_i^\vee,\ld\ra}T_i(z),$$
and, due to (P1),
$$\overline{T_i^{-1}(e_{n,-na_{ij}}\overline{z})} =(-1)^{\la\af_i^\vee,\ld\ra}\pi_i^{\la\af_i^\vee,\ld\ra+na_{ij}\choose2}q_i^{\la\af_i^\vee,\ld\ra-na_{ij}}T_i(\overline{e_{n,-na_{ij}}}z).$$
Hence,
$$T_i(\overline{e_{n,-na_{ij}}}z)=q_i^{-na_{ij}}\tJ_i^{np(j)}E_j^{(n)}T_i(z).$$
Now, by (P1), $p(i)a_{ij}$ is even, so
$$e'_{n,-na_{ij}}=\pi_i^{-na_{ij}\choose 2}q_i^{-na_{ij}}\overline{e_{n,-na_{ij}}}.$$
Hence, (b) holds once we observe that ${-na_{ij}\choose2}\equiv{na_{ij}\choose 2}$ (mod 2).

Finally, we prove (c) and (d). First note that $e_{n,m}=\om(\overline{f'_{n,m}})$ and $e'_{n,m}=\om(\overline{f_{n,m}})$. Now, if $z\in M_\ld$, then $f_{n,-na_{ij}}z,f'_{n,-na_{ij}}z\in M_{\ld+na_{ij}\af_i-n\af_j}$.

Using part (b) and Lemma \ref{L:OmegaT}(b),
\begin{align*}
T_i^{-1}(f_{n,-na_{ij}}z)&=\pi_i^{\la\af_i^\vee,\ld\ra+na_{ij}+1\choose2}\om^{-1}\left(\overline{T_i(e'_{n,-na_{ij}}\om(\overline{z}))}\right)\\
    &=\pi_i^{{na_{ij}\choose 2}+{na_{ij}\choose2}}\tJ_i^{np(j)} \om^{-1}(E_j^{(n)})\pi_i^{\la\af_i^\vee,\ld\ra+1\choose2}\om^{-1}\left(\overline{T_i(\om(\overline{z}))}\right).
\end{align*}
and, by part (a) and Lemma \ref{L:OmegaT}(c),
\begin{align*}
T_i(f'_{n,-na_{ij}}z)&=\pi_i^{\la\af_i^\vee,\ld\ra+na_{ij}+1\choose2}\om^{-1}\left(\overline{T_i^{-1}(e_{n,-na_{ij}}\om(\overline{z}))}\right)\\
    &=\pi_i^{{na_{ij}\choose 2}+{na_{ij}\choose2}}\tJ_i^{np(j)} \om^{-1}(E_j^{(n)})\pi_i^{\la\af_i^\vee,\ld\ra+1\choose2}\om^{-1}\left(\overline{T_i^{-1}(\om(\overline{z}))}\right).
\end{align*}
In both cases, we have used condition (P1) to deduce that
$${\la\af_i^\vee,\ld\ra+na_{ij}+1\choose2}\equiv{\la\af_i^\vee,\ld\ra+1\choose2}+{na_{ij}\choose 2}\;\;\;(\mbox{mod }2).$$
This proves (c) and (d).
\end{proof}

\subsection{The symmetries $T_i$ and $T_i^{-1}$ of $\UU$}

The properties of the braid operators on $\catO_{\rm int}$ allow us to define analogous operators on the quantum group itself.
In particular, Lemmas \ref{L:BraidsonEis}, \ref{L:BraidsandKs}, and \ref{L:BraidsonEjs} allow for
us to directly generalize the proof of \cite[\S37.2.3]{Lu}, obtaining the following theorem.

\begin{theorem}\label{T:BraidingOnU}
\begin{enumerate}
\item[(a)] For any $u\in \UU$, there exists a unique element $u'\in \UU$ such that $T_i(u'z)=uT_i(z)$ for any $M\in\catO_{\mathrm{int}}$ and any $z\in M$. Moreover, the map $u\mapsto u'$ is a $\Qqp$-algebra automorphism of $\UU$, 
denoted $T_i^{-1}$.
\item[(b)] For any $u\in \UU$, there exists a unique element $u''\in \UU$ such that $T_i^{-1}(u''z)=uT_i^{-1}(z)$ for any $M\in\catO_{\mathrm{int}}$ and any $z\in M$. Moreover, the map $u\mapsto u''$ is a 
$\Qqp$-algebra automorphism of $\UU$, denoted $T_i$.
\end{enumerate}
The automorphisms $T_i,T_i^{-1}:\UU\longrightarrow \UU$ are mutually inverse, and defined on the divided powers in the Chevalley generators of $\UU$ by the formulae:
\begin{align*}
\begin{array}{ll}
T_i(E_i^{(n)})=(-1)^n\pi_i^nq_i^{n(n-1)}\tJ_i^n \tK_{i}^{n}F_i^{(n)},&T_i^{-1}(E_i^{(n)})=(-1)^nq_i^{n(n-1)}F_i^{(n)}\tK_i^{-n},\\
T_i(F_i^{(n)})=(-1)^nq_i^{-n(n-1)}E_i^{(n)}\tK_i^{-n},&T_i^{-1}(F_i^{(n)})\!=\!(-1)^n\pi_i^nq_i^{-n(n-1)}\!\tJ_i^{n}\tK_{i}^{n}E_i^{(n)}\!\!\!,\\
T_i(E_j^{(n)})=\pi_i^{na_{ij}\choose2}\tJ_i^{np(j)}e_{i,j,n,-na_{ij}}, &T_i^{-1}(E_j^{(n)})=\pi_i^{na_{ij}\choose2}\tJ_i^{np(j)}e'_{i,j,n,-na_{ij}}, \\
T_i(F_j^{(n)})=\tJ_i^{np(j)}f_{i,j,n,-na_{ij}}, &T_i^{-1}(F_j^{(n)})=\tJ_i^{np(j)}f'_{i,j,n,-na_{ij}}, \\
T_i(K_\mu)=K_{s_i(\mu)}, &T_i^{-1}(K_\mu)=K_{s_i(\mu)},\\
T_i(J_\mu)=J_{s_i(\mu)},&T_i^{-1}(J_{\mu})=J_{s_i(\mu)},
\end{array}
\end{align*}
where the elements $e_{i,j,n,-na_{ij}}, e'_{i,j,n,-na_{ij}},f_{i,j,n,-na_{ij}},f'_{i,j,n,-na_{ij}}$ are defined in \eqref{E:eijnm}-\eqref{E:f'ijnm}.
\end{theorem}

\begin{remark}\label{R:parity of braiding operators} Observe that the above formulas imply that the braiding operators are \emph{even} automorphisms of $\UU$ (in the nontrivial cases, this follows from the fact that $\langle \alpha_i^\vee,\alpha_j\rangle$ is even when $i\in I_\one$). In contrast, the definition of $T_i$ as a map $M_\lambda\longrightarrow M_{s_i(\lambda)}$ implies that its parity is $\langle \alpha_i^\vee,\lambda\rangle$ as noted in \eqref{eq:parityofbraidoperators}.
\end{remark}

One may verify directly on the generators that
\begin{equation}\label{eq:braid vs sigma}
T_i\sigma=\sigma T_i^{-1}.
\end{equation}

Furthermore, by inspection of the images of the generators in Theorem \ref{T:BraidingOnU},
we see that $T_i^{\pm 1}$ preserve the integral form of $\UU$. In particular,
this implies the following corollary.

\begin{corollary}\label{C:TsonUap} The automorphisms $T_i$ and $T_i^{-1}$ of $\UU$ restrict to automorphisms of $\Uint$.
\end{corollary}

\begin{remark}In \cite{CFLW,C}, a modified form $\dot\UU$ of $\UU$
was defined \`a la Lusztig; to wit, one adds weight-space projections $1_\lambda$
for each $\lambda\in P$ to $\UU$ to obtain an algebra $\dot \UU$
on symbols $u1_\lambda$, where $u\in \UU$ and $\lambda\in P$, subject to some
natural relations. We note that, just as in \cite[\S 41.1]{Lu},
this modified form admits braiding operators $T_i^{\pm 1}$ ($i\in I$)
satisfying $T_i^{\pm 1}(u1_\lambda)=T_i^{\pm 1}(u)1_{s_i(\lambda)}$,
which restrict to automorphisms of the integral form of $\dot \UU$.
\end{remark}

\subsection{Braiding operators and comultiplication}
Let $M,N\in \catO^i_{\mathrm{int}}$. As usual, we regard $M\otimes N$ as a $\UU$-module via $\Delta$,
and note that $M\otimes N\in\catO^i_{\mathrm{int}}$. If $x\in M_t$ and $y\in N_s$, then $x\otimes y\in (M\otimes N)_{t+s}$ and
\[
\Delta(E_i)(x\otimes y)=E_ix\otimes y + \pi_i^{p(x)}(\pi_iq_i)^t x\otimes E_iy,\] 
\[\Delta(F_i)(x\otimes y)= q_i^{-s} F_ix\otimes y + \pi_i^{p(x)}x\otimes F_i y.
\]
Define operators $L_i',L_i'':M\otimes N\longrightarrow M\otimes N$ by
\begin{align}
\label{E:Li}L_i'(x\otimes y)=\sum_{n\geq 0}(-1)^n\pi_i^n(\pi_iq_i)^{n\choose 2}(\pi_i q_i-q_i^{-1})^n[n]^!_iF_i^{(n)}x\otimes E_i^{(n)}y,\\
\label{E:L_iinverse} L_i''(x\otimes y)=\sum_{n\geq0}(-1)^n\pi_i^n q_i^{-{n\choose 2}}(\pi_i q_i-q_i^{-1})^n[n]^!_iF_i^{(n)}x\otimes E_i^{(n)}y.
\end{align}
These operators are the precisely the operators $\overline{\Theta}$ and $\Theta$, respectively, for the algebra $\UU(i)$ defined in
\cite[\S 3.1]{CHW1}. In particular, we have the properties
\begin{equation}L_i'L_i''=L_i''L_i'=1:M\otimes N\longrightarrow M\otimes N,
\end{equation}
\begin{equation}\label{eq:LiCoprod}
L_i'\Delta(u)=\bar{\Delta}(u)L_i'
\end{equation}
We are therefore justified to introduce the new notations
$$L_i=L_i'\andeqn L_i^{-1}=L_i''.$$
We also note the following lemma follows from \eqref{eq:LiCoprod}.
\begin{lemma}\label{L:FiLi}
Let $x\in M_t$ and $y\in M_s$. Then,
$$\Delta(F_i)L_i^{-1}(x\otimes y)=L_i^{-1}((\pi_iq_i)^{s} F_ix\otimes y + \pi_i^{p(x)}x\otimes F_i y),$$
and
$$\Delta(E_i)L_i^{-1}(x\otimes y)=L_i^{-1}(E_ix\otimes y + \pi_i^{p(x)}q_i^{-t} x\otimes E_iy).$$
\end{lemma}

We will now relate the action of $T_i,T_i^{-1}$ on a tensor product of modules to their actions on each tensor factor.
Recall from \eqref{eq:parityofbraidoperators} that the maps $T_i$
are not generally homogeneous with respect to the $\Z/2\Z$-grading.
Consequently, when applying $T_i$ to the second tensor factor, we incur
a $\pi$ factor depending on the weight of the second tensor factor
and the parity of the first tensor factor. Therefore, we define
the maps $T_i\otimes T_i$ and $T_i^{-1}\otimes T_i^{-1}$ 
to be the $\Qqp$-linear maps defined by
\[(T_i\otimes T_i)(m\otimes n)=\pi_i^{sp(m)}T_i(m)\otimes T_i(n),\]
\[(T_i^{-1}\otimes T_i^{-1})(m\otimes n)=\pi_i^{sp(m)}T_i^{-1}(m)\otimes T_i^{-1}(n),\]
for $m\otimes n\in M_t\otimes N_s$. Note that $(T_i\otimes T_i)^{-1}(m\otimes n)=\pi_i^{st}T_i^{-1}\otimes T_i^{-1}(m\otimes n)$.
Then $L_i$ intertwines with the operators $T_i$ as follows.

\begin{lemma}\label{L:TiandLi} Let $M,N\in\catO^i_{\mathrm{int}}$. Then, for any $z\in M\otimes N$,
\begin{equation}\label{eq:TiandLi}
(T_i^{-1}\otimes T_i^{-1})\circ T_i\circ L_i^{-1}(z)=z.
\end{equation}
\end{lemma}

\begin{proof} First, we shall prove that \eqref{eq:TiandLi} holds for $z=x\otimes y\in M\otimes N$, then it also holds for $z'=\bar{\Delta}(F_i)z=(\pi_iq_i)^{s} F_ix\otimes y + x\otimes F_i y$. We may assume $x\in M_t$ and $y\in N_s$, and
so by assumption \[
 T_i\circ L_i^{-1}(z)=\pi_i^{st}(T_i\otimes T_i)(z).
\]
By Lemmas \ref{L:FiLi} and \ref{L:BraidsonEis}(a),
\begin{align*}
T_i(L_i^{-1}(z'))&=T_i(\Delta(F_i)L_i^{-1}(z))=-q_i^{t+s}\Delta(E_i)T_i(L_i^{-1}(z)).
\end{align*}
On the other hand, we have
\begin{align*}
(T_i\otimes T_i)(z')&=(T_i\otimes T_i)((\pi_iq_i)^sF_ix\otimes y+\pi_i^{p(x)}x\otimes F_iy)\\
&=(\pi_iq_i)^s\pi_i^{(1+p(x))s}T_i(F_ix)\otimes T_i(y)+\pi_i^{(s+1)p(x)}T_i(x)\otimes T_i(F_iy)\\
    &=\pi_i^{sp(x)}((-q_i^{s+t}E_iT_i(x)\otimes T_i(y))+\pi_i^{p(x)}T_i(x)\otimes(-q_i^sE_iT_i(y))\\
    &\stackrel{\rm(a)}{=}-q_i^{s+t}\pi_i^{sp(x)}(E_i\otimes 1+q_i^{-t}\pi_i^{t}1\otimes E_i)(T_i(x)\otimes T_i(y))\\
    &\stackrel{\rm(b)}{=}-q_i^{s+t}\Delta(E_i)(T_i\otimes T_i)(x\otimes y).
\end{align*}
We note that the equality (a) follows from $p(T_i(x))=tp(i)+p(x)$, whereas (b) follows since $T_i(x)\in M_{-t}$. Then
applying the induction hypothesis, we have shown that $T_i(L_i^{-1}(z'))=\pi^{st}(T_i\otimes T_i)(z')$;
since $\pi_i^{st}=\pi_i^{(s-2)t}=\pi_i^{s(t-2)}$, we have \[
(T_i^{-1}\otimes T_i^{-1})\circ T_i\circ L_i^{-1}(z')=z',
\]
and thus the claim is proved.

Following \cite[Proposition 5.3.4]{Lu}, define $Z_\ell\subset M\otimes N$, $\ell\geq0$,
to be the subspace spanned by vectors of the form $x\otimes F^{(q)}y$ with $E_iy=0$.
Then $M\otimes N=\sum_{\ell\geq0}Z_\ell$. We prove (a) by induction on $\ell$. In particular, we will show:
\begin{enumerate}
\item[(1)] The identity \eqref{eq:TiandLi} holds for $z\in Z_0$, and
\item[(2)] If the identity \eqref{eq:TiandLi} holds for $z\in Z_\ell$, then it holds for $z\in Z_{\ell+1}$.
\end{enumerate}
The proof of (2) is exactly as in \cite{Lu}.
Indeed, assume \eqref{eq:TiandLi} holds for $z\in Z_\ell$ and write $z=x\otimes F_i^{(\ell)}y$ with $E_iy=0$.
Then, by the discussion above, \eqref{eq:TiandLi} also holds for $z'=(\pi_iq_i)^sF_ix\otimes F_i^{(\ell)}y +[\ell+1]_ix\otimes F_i^{(\ell+1)}y$.
Since, by assumption \eqref{eq:TiandLi} holds for $z''=(\pi_iq_i)^sF_ix\otimes F_i^{(\ell)}y$,
it holds for $z'''=[\ell+1]_ix\otimes F_i^{(\ell+1)}y$ as well. Since $[\ell+1]_i\neq 0$, (2) is proved.

We now prove the base case (1). To this end, assume $z=x\otimes y\in M_m\otimes N_n$ with $E_iy=0$.
Then by the definition of $L_i$, $L_i(z)=z$ and so it remains to show that $T_i(z)=\pi^{mn}(T_i\otimes T_i)z$.
By Lemma \ref{L:BarT}, we deduce that
\begin{align*}
T_i(z)&=(-1)^{m+n}\pi_i^{m+n+1\choose 2}q_i^{-(m+n)}\overline{T_i^{-1}(\overline{z})}\\
    &=(-1)^{m+n}(\pi_iq_i)^{-(m+n)}\hspace{-1em}\sum_{-a+b-c=m+n}\hspace{-1em}(-1)^b\pi_i^{c+b+a(m+n)}q_i^{-ac+b}E_i^{(a)}F_i^{(b)}E_i^{(c)}z.
\end{align*}
Let $\bigstar=(-1)^{m+n}(\pi_iq_i)^{m+n}T_i(x\otimes y)$.
Using \eqref{eq:Tidef} and \eqref{eq:coprodDivPow}, we compute
\begin{align*}
\bigstar&=\sum_{-a'-a''+b'+b''-c=m+n}(-1)^{b'+b''}\pi_i^{c+b'+b''+a'(m+n)+a''(m+n)}q_i^{-a'c-a''c+b'+b''}\\
&\hspace{1em}\times q_i^{a'a''}(E_i^{(a')} (\tJ_i\tK_i)^{a''}\otimes E_i^{(a'')})(\pi_iq_i)^{-b'b''}(F_i^{(b')}\otimes \tK_i^{-b'}F_i^{(a'')})(E_i^{(c)}x\otimes y)\\&\\
&=\sum_{-a'-a''+b'+b''-c=m+n}(-1)^{b'+b''}q_i^{-a'c-a''c+b'+b''+a'a''+b'b''-b'n+a''(m+2c-2b')}\\
&\hspace{4em}\times \pi_i^{c+b'+b''+a'(m+n)+a''(m+n)+b'b''+b''c+(a''+b'')p(x)+a''b'+a''c+a''m}\\
&\hspace{4em}\times \pi_i^{a''b''+{a''+1\choose 2}}\left[a''-b''+n\atop a''\right]_i(E_i^{(a')}F_i^{(b')}E_i^{(c)}x\otimes F_i^{(b''-a'')}y)
\end{align*}
Now, make the substitution $b''=a''+g$. We note that, since $E_iy=0$ and $y\in N_n$, $F_i^{(g)}y=0$ for $g>n$ and
hence the sum above is nonzero only when $g\leq n$. Then we may rewrite the above to obtain
\begin{align*}
\bigstar&=\sum_{\substack{-a'+b'-c=m+n-g\\g\leq n}}(-1)^{b'+g}\pi_i^{c+b'+a'(m+n+g)+(m+n)g+gp(x)}q_i^{-a'c+b'+g+b'g-b'n}\\
&\times\!\!\left(\sum_{a''\geq 0}(-1)^{a''}\pi_i^{a''(n+g)+{a''+1\choose2}}q_i^{a''(1+(n-g))}\left[n-g\atop a''\right]_i\right)\!E_i^{(a')}\!F_i^{(b')}\!E_i^{(c)}\!x\otimes F_i^{(g)}\!y\\
\end{align*}
Now using the image of the identity \eqref{eq:binomidf} under $\bar{\phantom{x}}$, it follows that the sum over $a''$ is zero unless $n-g=0$.
Now multiplying $\bigstar$ by $(-1)^{m+n}(\pi_iq_i)^{-m-n}$, we have
\begin{align*}
&T_i(x\otimes y)=\pi_i^{mn+np(x)}(-1)^m(\pi_iq_i)^m\pi_i^{m\choose2}\\
&\times\left(\sum_{-a'+b'-c=m}(-1)^{b'}\pi_i^{c+b'+a'm+{m\choose2}}q_i^{-a'c+b'}E_i^{(a')}F_i^{(b')}E_i^{(c)}x\right) \otimes F_i^{(n)}y\\
&=\pi_i^{mn+np(x)}T_i(x)\otimes T_i(y)=\pi_i^{mn}(T_i\otimes T_i)(x\otimes y).
\end{align*}
This completes the proof.
\end{proof}

\begin{corollary}\label{C:TiandComult} 
Let $M,N\in  \catO_{\rm int}$ and $u\in\UU$. Then as maps on $M\otimes N$,
$$L_i\Delta(u)L_i^{-1}=(T_i^{-1}\otimes T_i^{-1})\Delta(T_i(u)).$$
\end{corollary}

\begin{proof}
Let $z=x\otimes y\in M_\lambda\otimes N_\mu$. Inserting $L_i(z)$ into Lemma \ref{L:TiandLi}, we deduce that
$$(T_i^{-1}\otimes T_i^{-1})T_i(z)=L_i(z).$$
Therefore, for $u\in \UU$,
\begin{align*}
L_i\Delta(u)L_i^{-1}(z)&=L_i(uL_i^{-1}(z))\\
    &=L_i(u T_i^{-1}T_i(L_i^{-1}(z)))\\
    &=\pi_i^{\langle \alpha_i^\vee,\lambda\rangle\langle \alpha_i^\vee,\mu\rangle}L_i(uT_i^{-1}(T_i\otimes T_i)(z))\\
    &=\pi_i^{\langle \alpha_i^\vee,\lambda\rangle\langle \alpha_i^\vee,\mu\rangle}(T_i^{-1}\otimes T_i^{-1})T_i(uT_i^{-1}(T_i\otimes T_i)(z))\\
    &=\pi_i^{\langle \alpha_i^\vee,\lambda\rangle\langle \alpha_i^\vee,\mu\rangle}(T_i^{-1}\otimes T_i^{-1})(T_i(u)(T_i\otimes T_i)(z))\\
    &=(T_i^{-1}\otimes T_i^{-1})\Delta(T_i(u))(z).
\end{align*}
This proves the corollary.
\end{proof}
\section{Braid group action and the inner product}

\subsection{Algebras $\UU^0_J$ and $\UU^+_J$}\label{SS:UPJ}

Recall that the
$\Qqp$-algebra $\UU^0$ has a basis \[\{J_\mu K_\nu \mid \mu,\nu \in
P^\vee\}.\] Denote by $\UU^0_J$ the $\Qqp$-subalgebra of $\UU^0$
generated by $J_i$ for $i \in I_\one$ (or equivalently, generated by
$\tJ_i$ for $i \in I$). Then clearly $\UU^0_J$ is a free
$\Qqp$-module with basis $\{J_\nu \mid \nu \in \sum_{i\in I_\one}
\Z \alpha_i^\vee\}$. Moreover, note that we can view $\UU$ as an algebra over $\UzJ$
which, by the triangular decomposition, is free as a $\UzJ$-module.
We note that the braid operators are $\UzJ$-linear maps by Lemma \ref{L:BraidsandKs} and Theorem \ref{T:BraidingOnU}.

Denote by $\UU^+_J$ the $\Qqp$-subalgebra of $\UU$
generated by $E_i, \tJ_i$ $(i\in I)$, or equivalently, generated
by the subalgebras $\UU^+$ and $\UU^0_J$.
We can endow $\UU^+_J$ with a twisted bialgebra structure analogous to $\ff$.
We transport the maps $\ir,\ri:\ff\to\ff$ as follows.
Define $q$-derivations $\ir$ and $\ri$ on $\UpJ$ by $\ir(\tJ_\mu x)=\tJ_\mu\ir(y)^+$ and $\ri(\tJ_\nu x)=\tJ_\nu\ri(y)^+$
if $y\in\ff$ satisfies $y^+=x\in\Up$, and $\nu\in Q_+$.
Next, define $$r:\UpJ\longrightarrow\UpJ\otimes\UpJ$$ by $r(x)=\sum y_{(1)}^+\otimes y_{(2)}^+$
if $y\in\ff$ satisfies $y^+=x$ and $r(y)=\sum y_{(1)}\otimes y_{(2)}$, and $r(\tJ_\nu)=\tJ_\nu\otimes\tJ_\nu$ for all $\nu\in Q_+$.
Then, $r$ is an algebra homomorphism with respect to the twisted multiplication \eqref{eq:twistedmul}.
Moreover, recall the anti-automorphism $\sigma:\UU\rightarrow \UU$
defined in \eqref{eq:rhodef}.
Then for $x\in \UpJ$ with $r(x)=\sum x_1\otimes x_2$, we have
\begin{equation}\label{eq:r vs sigma}
r(\sigma(x))=\sum \sigma(x_2)\otimes \sigma(x_1).
\end{equation}
In particular, $r_i\circ\sigma=\sigma\circ{}_i r$.

Finally, define a bilinear form $(\cdot,\cdot):\UpJ\otimes\UpJ\longrightarrow\UzJ$ by
\[(\tJ_{\nu_1}x_1,\tJ_{\nu_2}x_2)=\tJ_{\nu_1+\nu_2}(y_1,y_2)\;\;\;\mbox{if }y_1^+=x_1,\; y_2^+=x_2,\;\mbox{ and }\nu_1,\nu_2\in Q_+.\]
We note that, from the definitions, analogues of \eqref{eq:bilform} and \eqref{eq:derivsInnerProd} hold for this bilinear form.

\subsection{The algebras $\UU^+_J[i]$ and ${}^\sigma\UU^+_J[i]$}
Fix $i\in I$, and for any $j\in I\backslash\{i\}$ set
\begin{equation}\label{eq:e(i,j;m)}
\begin{aligned}
e(i,j;m)&=\sum_{r+s=m}(-1)^r\pi_i^{p(r;i,j)}(\pi_iq_i)^{-r(a_{ij}+m-1)}E_i^{(r)}E_jE_i^{(s)}\in\Up;\\
e'(i,j;m)&=\sum_{r+s=m}(-1)^r\pi_i^{p(r;i,j)}(\pi_iq_i)^{-r(a_{ij}+m-1)}E_i^{(s)}E_jE_i^{(r)}\in\Up;
\end{aligned}
\end{equation}
where $p(r;i,j)=p(r,1;i,j)=rp(i)p(j)+\binom{r}{2}p(i)$.
Then $e(i,j;m)=e_{i,j;1,m}$ and $e'(i,j;m)=e'_{i,j;1,m}$.

Let $\UpJ[i]$ (resp. ${}^\sigma\UpJ[i]$) be the $\UU^0_J$-subalgebra of $\UU^+_J$ generated by $e(i,j;m)$ (resp. $e'(i,j;m)$) for $m\geq 0$ and $j\in I\backslash\{i\}$. Since $\sigma(e(i,j;m))=e'(i,j;m)$, we have $\sigma(\UpJ[i])={}^\sigma\UpJ[i]$.

\begin{lemma}\label{L:ffi1} (a) $\UpJ=\sum_{t\geq0}E_i^{t}\UU^+_J[i]=\sum_{t\geq 0}\UU^+_J[i] E_i^t$;

(b) $\UU^+_J=\sum_{t\geq0}{}^\sigma\UU^+_J[i]E_i^t=\sum_{t\geq0}E_i^t{}^\sigma\UU^+_J[i]$.
\end{lemma}

\begin{proof} Clearly (b) follows from (a) by applying $\sigma$. To prove (a), note that Lemma \ref{L:HigherSerre} provides the relation
$$\displaystyle e(i,j;m)E_i-q_i^{-a_{ij}-2m}\pi_i^{m+np(j)}E_ie(i,j;m)=[m+1]_ie(i,j;m+1).$$
Therefore, given any product $y_1\cdots y_n$ in which each factor is either $E_i$ or one of the $e(i,j;m)$, we may use this relation to rewrite it either as a linear combination of products of the form $E_i^t y_{1}'\cdots y_k'$, where $y_{1}',\ldots,y_k'\in\UpJ[i]$, or as a linear combination of products of the form $y_1''\cdots y_k''E_i^t$, where $y_1'',\ldots,y_k''\in\UpJ[i]$. Now, the result follows from the fact that $\UpJ$ is generated by $\UU^0_J$, together with $E_i$ and $E_j=e(i,j;0)$ for $j\neq i$.
\end{proof}

\begin{lemma}\label{L:Braidsones}
Assume $i,j\in I$, $i\neq j$. For any $0\leq m\leq-a_{ij}$,
\begin{enumerate}
\item[(a)] $T_i(e'(i,j;m))
    =\pi_i^{a_{ij}\choose 2}\pi_i^{(p(j)+1)(-a_{ij}-m)}\tJ_i^{p(j)}e(i,j;-a_{ij}-m)$;
\item[(b)] $T_i^{-1}(e(i,j;m))
    =\pi_i^{a_{ij}\choose2}\pi_i^{(p(j)+1)(-a_{ij}-m)}\tJ_i^{p(j)}e'(i,j;-a_{ij}-m)$.
\end{enumerate}
\end{lemma}

\begin{proof}
The statements are equivalent by Proposition~ \ref{P:inversebraids}.
We prove (a) by downward induction on $m$, the initial case $m=-\ang{i,j'}$
being Lemma \ref{L:BraidsonEjs}(b). To this end, recall that by
Lemma \ref{L:HigherSerre}
$$-F_ie(i,j;m)+\pi_i^{m+p(j)}e(i,j;m)F_i
 =[-na_{ij}-m+1]_i\pi_i^{p(j)+1}\tJ_i^{-1}e(i,j;m-1).$$
Applying the anti-automorphism $\sigma$, we obtain the equation
$$\pi_i\tJ_i(-e'(i,j;m)F_i+\pi_i^{m+p(j)}F_ie'(i,j;m))
\!=\![-na_{ij}-m+1]_i\pi_i^{p(j)+1}e'(i,j,m-1)\tJ_i.$$ Applying
$T_i$ to both sides, and applying the induction hypothesis together
with Lemma \ref{L:HigherSerre} and Theorem \ref{T:BraidingOnU},
we have
\begin{align*}
[-n  a_{ij}&-m+1]_i\pi_i^{p(j)+1}T_i(e'(i,j;m-1))\tJ_i^{-1}\\
    &=\pi_i^{a_{ij}\choose 2}\pi_i^{(p(j)+1)(-a_{ij}-m)}\tJ_i^{p(j)}(e(i,j;-a_{ij}-m)E_i
    \\&\hspace{2em}-q_i^{-a_{ij}-2m}\pi_i^{m+p(j)}E_ie(i,j;-a_{ij}-m))\tJ_i^{-1}\\
    &=\pi_i^{a_{ij}\choose 2}\pi_i^{(p(j)+1)(-a_{ij}-m)}\tJ_i^{p(j)}[-na_{ij}-m+1]e(i,j;-a_{ij}-m+1)\tJ_i^{-1}.
\end{align*}
Therefore, (a) follows.
\end{proof}

The next lemma is a consequence of Lemma \ref{L:Braidsones}.

\begin{lemma}\label{L:ffi2} The braiding operator $T_i^{-1}$ defines an isomorphism of $\UpJ[i]$ onto ${}^\sigma\UpJ[i]$ with $T_i$ being the inverse isomorphism.
\end{lemma}

\begin{lemma}\label{L:ffi3} Assume that $x\in\UpJ$ satisfies $T_i^{-1}(x)\in\UpJ$. Then $\ir(x)=0$.
\end{lemma}

\begin{proof}
By Proposition \ref{prop:EFcommutation}, we have for homogeneous $x\in\UpJ$,
\begin{equation}\label{eq:comform}
xF_i-\pi_i^{p(x)}F_ix=\frac{\pi_i^{p(x)-p(i)}\tJ_i\tK_i\ \ir(x)-\ri(x)\tK_{-i}\,\,
\, }{\pi_iq_i-q_i^{-1}}
\end{equation}
Using Lemma \ref{L:ffi1}, we may write
$$\frac{\ir(x)}{\pi_iq_i-q_i^{-1}}=\sum_{t\geq 0}E_i^{(t)}y_t$$
and
$$\frac{\ri(x)}{\pi_iq_i-q_i^{-1}}=\sum_{t\geq 0}E_i^{(t)}z_t$$
where $y_t,z_t\in\UpJ[i]$ are homogeneous. Using Lemma \ref{L:ffi2}, we have \[T_i^{-1}(y_t),T_i^{-1}(z_t)\in\UpJ\] for all $t\geq0$. Therefore, applying $T_i^{-1}$ to \eqref{eq:comform}, be obtain
\begin{equation}\label{eq:Ticommform}\begin{aligned}
&-\pi_i\tJ_i(T_i^{-1}(x)\tK_iE_i-\tK_iE_iT_i^{-1}(x))\\
    &=\sum_{t\geq 0}(-1)^tq_i^{t(t-1)}F_i^{(t)}\tK_{-ti}\bigg(\pi_i^{p(x)-p(i)}\tJ_i\tK_{-i}T_i^{-1}(y_t)-T_i^{-1}(z_t)\tK_i\bigg).
\end{aligned}\end{equation}
By assumption, the left-hand side of \eqref{eq:Ticommform} is in $\tK_{i}\UpJ$, hence so is the right-hand side.
Using the triangular decomposition of $\UU$, we deduce that $T_i^{-1}(y_t)=0$ for all $t\geq 0$, and $T_i^{-1}(z_t)=0$ for all $t>0$.
As $T_i^{-1}$ is an automorphism of $\UU$, we deduce that $y_t=0$ for all $t\geq0$ proving the claim
(note, however, that we may have $z_0\neq0$).
\end{proof}

\begin{lemma}\label{L:ffi4} Let $x_t\in\UpJ$, $t\geq 1$ belong to $\ker(\ir)$, where only finitely many  are nonzero.
Assume that $\sum_{t\geq0}E_i^{(t)}x_t=0$ or $\sum_{t\geq0}x_tE_i^{(t)}=0$. Then, $x_t=0$ for all $t$.
\end{lemma}

\begin{proof} Assume $x_t=0$ for $t>N$. We prove the proposition by induction on $N$. If $N=0$, then the lemma is trivially true. Assume $N>0$. Then, using the fact that $x_t\in\ker(\ir)$, we have
\[0=\ir^N\left(\sum_{t\geq0}E_i^{(t)}x_t\right)=q_i^{N\choose 2}\,x_N,\text{ or }\]
\[ 0=\ir^N\left(\sum_{t\geq0}x_tE_i^{(t)}\right)=q_i^{N(\alpha_i,|x_N|)+\binom{N}{2}}\,x_N.\]
In particular, $x_N=0$ and induction applies.
\end{proof}

\begin{proposition}\label{P:ffi} (a) The following three subspaces coincide:
$$\UpJ[i]=\{x\in\UpJ\mid T_i^{-1}(x)\in\UpJ\}=\{x\in\UpJ\mid \ir(x)=0\}.$$

(b) The following three subspaces coincide:
$${}^\sigma\UpJ[i]=\{x\in\UpJ\mid T_i(x)\in\UpJ\}=\{x\in\UpJ\mid \ri(x)=0\}.$$
\end{proposition}

\begin{proof} We obtain (b) from (a) by applying $\sigma$ using \eqref{eq:braid vs sigma} and \eqref{eq:r vs sigma}.
To prove (a), note that by Lemmas \ref{L:ffi2} and \ref{L:ffi3}, we have
$$\UpJ[i]\subseteq\{x\in\UpJ|T_i^{-1}(x)\in\UpJ\}\subseteq\{x\in\UpJ|\ir(x)=0\}.$$
Now, assume that $x\in\UpJ$ satisfies $\ir(x)=0$. By Lemma \ref{L:ffi1}, we may write $x=\sum_{t\geq0}E_i^{(t)}x_t$ where the $x_t$ belong to the kernel of $\ir$. Then, the sum
$$0=(x_0-x)+\sum_{t\geq 1}E_i^{(t)}x_t$$
satisfies the conditions of Lemma \ref{L:ffi4}. In particular, $x-x_0=0$, so $x\in\UpJ[i]$. This completes the proof.
\end{proof}

Combining Lemma \ref{L:ffi4} and Proposition \ref{P:ffi} yields the following refinement of Lemma \ref{L:ffi1}.
\begin{corollary}\label{cor:direct sum UpJ} The following $\Qqp$-module decompositions hold.
\begin{enumerate}
\item[(a)] $\UpJ=\bigoplus_{t\geq0}E_i^{t}\UU^+_J[i]=\bigoplus_{t\geq 0}\UU^+_J[i] E_i^t$, and in particular
\[\UpJ=E_i \UpJ \oplus \UpJ[i]= \UpJ E_i\oplus \UpJ[i].\]
\item[(b)] $\UU^+_J=\bigoplus_{t\geq0}{}^\sigma\UU^+_J[i]E_i^t=\bigoplus_{t\geq0}E_i^t{}^\sigma\UU^+_J[i]$, and in particular
\[\UpJ=\UpJ E_i\oplus {}^\sigma\UpJ[i]=E_i \UpJ \oplus {}^\sigma\UpJ[i].\]
\end{enumerate}
\end{corollary}

Recall the map $r:\UpJ\rightarrow \UpJ\otimes \UpJ$ defined in
\SS \ref{SS:UPJ}.

\begin{lemma}\label{L:reijm}
Let $P(i,j;m;t)=\prod_{h=0}^{m-t-1}(1-\pi_i^{h+1-m}q_i^{2h+2-2m-2a_{ij}})$. Then we have the following identities.
\begin{enumerate}
\item[(a)]$r(e(i,j;m))=1\otimes e(i,j;m) +\sum_{t=0}^m(\pi_iq_i)^{t(m-t)} P(i,j;m;t)e(i,j;t)\otimes E_i^{(m-t)}.$
\item[(b)]$r(e'(i,j;m))=e'(i,j;m)\otimes1  +\sum_{t=0}^m(\pi_iq_i)^{t(m-t)} P(i,j;m;t)E_i^{(m-t)}\otimes e(i,j;t).$
\end{enumerate}
\end{lemma}

\begin{proof}
Using the fact that $r$ is an algebra homomorphism along with
\eqref{eq:r on divpow}, we have
\begin{align*}
r(e(i,&j;m))=\sum(-1)^{r'+r''}\pi_i^{p(r'+r'';i,j)}(\pi_iq_i)^{-(r'+r'')(a_{ij}+m-1)-r'r''-s's''}\\
&\hspace{5em}\times(E_i^{(r')}\otimes E_i^{(r'')})(E_j\otimes 1+1\otimes E_j)(E_i^{(s')}\otimes E_i^{(s'')})\\
&=\sum(-1)^{r'+r''}\pi_i^{p(r'+r'';i,j)+r''p(j)+r''s'}(\pi_iq_i)^{-(r'+r'')(a_{ij}+m-1)-r'r''-s's''}\\
\tag{c}&\hspace{2em}\times q_i^{-r''a_{ij}-2r''s'}\left[{r''+s''\atop r''}\right]_iE_i^{(r')}E_jE_i^{(s')}\otimes E_i^{(r''+s'')}\\
&\hspace{1em}+\sum(-1)^{r'+r''}\pi_i^{p(r'+r'';i,j)+s'p(j)+r''s'}(\pi_iq_i)^{-(r'+r'')(a_{ij}+m-1)-r'r''-s's''}\\
\tag{d}&\hspace{2em}\times q_i^{-s'a_{ij} -2r''s'}\left[{r'+s'\atop r'}\right]_iE_i^{(r'+s')}\otimes E_i^{(r'')}E_jE_i^{(s'')}
\end{align*}
where the sums are all over $r'+r''+s'+s''=m$.

Consider the sum (d). We note that the power of $\pi_i$ in (d) is
\[p(r'+r'';i,j)+s'p(j)+r''s'=(r'+s'+r'')p(j)+\binom{r'}{2}+\binom{r''}{2}+r''(s'+r').\]
Writing $r'+s'=t$ and $r''+s''=m-t$, we have
\[p(r'+r'';i,j)+s'p(j)+r''s'=p(r'';i,j)+\binom{r'}{2}+tp(j)+tr''.\]
Similarly, since $s'a_{ij}-2r''s'\in 2\Z$, the power of $\pi_iq_i$ in (d) is $\heartsuit$, where
\begin{align*}
\heartsuit&=-(r'+r'')(a_{ij}+m-1)-r'r''-s's''-s'a_{ij} -2r''s'\!\\
&=-r''(a_{ij}+m-1)-(r'+s')(r''+a_{ij})-r'm+r'-s'(r''+s'')\\
&=-r''(a_{ij}+m-1)-t(r''+a_{ij}))-r'm+r'-s'(m-t)\\
&=-r''(a_{ij}+m+t-1)-ta_{ij}-(s'+r')m+s't+r'\\
&=-r''(a_{ij}+m+t-1)-t(a_{ij}+m)+s'(t-1)+s'+r'\\
&=-r''(a_{ij}+m+t-1)-t(a_{ij}+m-1)+s'(t-1)\\
&=-r''(a_{ij}+m+t-1)-t(a_{ij}+m-1)+(t-r')(t-1)\\
&=-r''(a_{ij}+m+t-1)-t(a_{ij}+m-1)+2\binom{t}{2}-r'(t-1)
\end{align*}
Therefore, we can rewrite (d) as
\begin{align*}
&\sum_{t=0}^m\sum_{r''+s''=m-t}(-1)^{r''}\pi_i^{p(r'';i,j)+tp(j)+tr''}(\pi_iq_i)^{-r''(a_{ij}+m+t-1)-t(a_{ij}+m-1)+2\binom{t}{2}}\\
&\tag{e}\times\sum_{r'+s'=t}(-1)^{r'}\pi_i^{{r'\choose 2}}(\pi_iq_i)^{-r'(t-1)}\left[{t\atop r'}\right]_iE_i^{(t)}\otimes E_i^{(r'')}E_jE_i^{(s'')}.
\end{align*}
Applying the bar involution to \eqref{eq:binomidf}, we conclude that the sum over $r'+s'$ is 0 unless $t=0$.
Hence, (e) becomes
\begin{align*}
1\otimes\sum_{r''+s''=m}(-1)^{r''}\pi_i^{p(r'';i,j)}(\pi_iq_i)^{-r''(a_{ij}+m-1)} E_i^{(r'')}E_jE_i^{(s'')}=1\otimes e(i,j;m).
\end{align*}
Next, rewrite the sum corresponding to (c) in a similar manner to (d) to obtain
\begin{align*}
&\sum_{t=0}^m\sum_{r''+s''=m-t}(-1)^{r''}\pi_i^{\binom{r''}{2}+r''t}(\pi_iq_i)^{-r''(2a_{ij}+m+t-1)}\left[{m-t\atop r''}\right]_i\\
&\times\sum_{r'+s'=t}(-1)^{r'}\pi_i^{p(r';i,j)}(\pi_iq_i)^{-r'(a_{ij}+t-1)-(m-t)t} E_i^{(r')}E_jE_i^{(s')}\otimes E_i^{(m-t)}\\
&=\sum_{t=0}^m(\pi_iq_i)^{-(m-t)t}\sum_{r''+s''=m-t}\pi_i^{\binom{r''}{2}}q_i^{r''(m-t-1)}\left[{m-t\atop r''}\right]_i(-\pi_i^{1-m}q_i^{2-2m-2a_{ij}})^{r''}\\
&\times\sum_{r'+s'=t}(-1)^{r'}\pi_i^{p(r';i,j)}(\pi_iq_i)^{-r'(a_{ij}+t-1)} E_i^{(r')}E_jE_i^{(s')}\otimes E_i^{(m-t)}\\
&=\sum_{t=0}^m(\pi_iq_i)^{-(m-t)t}\sum_{r''+s''=m-t}\pi_i^{\binom{r''}{2}}q_i^{r''(m-t-1)}\left[{m-t\atop r''}\right]_i(-\pi_i^{1-m}q_i^{2-2m-2a_{ij}})^{r''}\\
&\times e(i,j;t)\otimes E_i^{(m-t)}
\end{align*}
By evaluating the identity \eqref{eq:binomidc} at $z=-\pi_i^{1-m}q_i^{2-2m-2a_{ij}}$, we have
\[\begin{aligned}
&\sum_{r''+s''=m-t}\pi_i^{\binom{r''}{2}}q_i^{r''(m-t-1)}\left[{m-t\atop r''}\right]_i(-\pi_i^{1-m}q_i^{2-2m-2a_{ij}})^{r''}\\
&=\prod_{h=0}^{m-t-1}(1-\pi_i^{h+1-m}q_i^{2h+2-2m-2a_{ij}})
\end{aligned}\]
which proves (a).

Finally, (b) follows from (a) since $r(\sigma(x))=(\sigma\otimes\sigma)^tr(x)$.
\end{proof}

\begin{lemma}\label{L:ffi5} Let $x\in\UpJ[i]$, and let $y=T_i^{-1}(x)\in{}^\sigma\UpJ[i]$ (see Lemma \ref{L:ffi2}).
We have $r(x)\in\UpJ[i]\otimes\UpJ$ and $r(y)\in{}^\sigma\UpJ\otimes\UpJ[i]$.
\end{lemma}

\begin{proof} Observe that if the lemma holds for $x_1$ and $x_2$ (resp. $y_1$ and $y_2$), then it holds for $x_1x_2$ (resp. $y_1y_2$) since, after twisting multiplication in $\UpJ\otimes\UpJ[i]$ (resp. ${}^\sigma\UpJ[i]\otimes\UpJ$), $r$ is multiplicative and $\UpJ[i]$ (resp. ${}^\sigma\UpJ[i]$) is closed under multiplication. Therefore, it is enough to check the lemma for $x=e(i,j;m)$ (resp. $y=e'(i,j;m)$). For these elements, the result follows from Lemma \ref{L:reijm}.
\end{proof}

Let $x\in\UpJ[i]$ and $y=T_i^{-1}(x)\in{}^\sigma\UpJ[i]$.
Using the decomposition $\UpJ=\UpJ E_i\oplus\UpJ[i]$, let $'r(x)\in\UpJ[i]\otimes\UpJ[i]$ be the unique element such that
\begin{align}\label{E:rprime}
r(x)-{'r}(x)\in\UpJ[i]\otimes \UpJ E_i.
\end{align}
Using the decomposition $\UpJ=E_i\UpJ \oplus {}^\sigma\UpJ[i]$, let $''r(y)\in{}^\sigma\UpJ[i]\otimes{}^\sigma\UpJ[i]$ be the unique element such that
\begin{align}\label{E:rprimeprime}
r(y)-{''r}(y)\in E_i\UpJ \otimes{}^\sigma\UpJ[i].
\end{align}

\begin{lemma}\label{L:ffi6}
We have $(T_i^{-1}\otimes T_i^{-1})('r(x))={''r}(T_i^{-1}(x)).$
\end{lemma}

\begin{proof}
Set $y=T_i^{-1}(x)$ as above. Let $\{u_h\}_{h\in H}$ be a homogeneous $\UzJ$-basis for $\UpJ[i]$. For each $h\in H$, set $v_h=T_i^{-1}(u_h)$, so $\{v_h\}_{h\in H}$ is a basis for ${}^\sigma\UpJ[i]$. Then, by Lemma \ref{L:ffi5}, we may uniquely write
\begin{align*}
r(x)&=\sum_{n\geq0;h,h'\in H}c(n;h,h')u_h\otimes u_{h'}E_i^{(n)}\\
r(y)&=\sum_{n\geq0;h,h'\in H}d(n;h,h') E_i^{(n)}v_h  \otimes v_{h'}
\end{align*}
where $c(n;h,h'),d(n;h,h')\in\UU^0_J\otimes\UU^0_J$ are zero for all but finitely many indices. Note that we have
\begin{align*}
{'r}(x)&=\sum_{h,h'\in H}c(0;h,h')u_h\otimes u_{h'}\\
{''r}(y)&=\sum_{h,h'\in H}d(0;h,h') v_h \otimes v_{h'}.
\end{align*}
Then the lemma will follow once we show that
$c(0;h,h')=d(0;h,h'),$
for all $h,h'\in H$.
Write
$$c'(n;h,h')=\pi^{(p(u_{h'})+np(i))p(u_{h})}q^{(|u_{h'}|+n\alpha_i, |u_h|)}c(n;h,h')$$
and
$$d'(n;h,h')=\pi^{(p(v_{h'})+np(i))p(v_{h})}q^{(|v_{h'}|+n\alpha_i, |v_h|)}d(n;h,h').$$
Note that $p(v_h)=p(u_h)$ and $|v_{h}|=s_i(|u_h|)$. Since $(-,-)$ on $Q$ is $W$-invariant,
$c(0;h,h')=d(0;h,h')$ if and only if $c'(0;h,h')=d'(0;h,h')$
for all $h,h'\in H$.

Using \eqref{eq:coprod f vs U}, we have
\begin{align*}
\tag{a}\Delta(x)&=\sum_{n\geq0;h,h'\in H}c'(n;h,h')
u_h\tJ_{|u_{h'}|+n\alpha_i^\vee}\tK_{|u_{h'}|+n\alpha_i^\vee}\otimes u_{h'}E_i^{(n)} \\
\tag{b}\Delta(y)&=\sum_{n\geq0;h,h'\in H}d'(n;h,h') E_i^{(n)}v_h \tJ_{|v_{h'}|}\tK_{|v_{h'}|}\otimes  v_{h'}.
\end{align*}
Then, $\Delta(x)=\Delta(T_i(y))$ by definition. Therefore, applying $T_i^{-1}\otimes T_i^{-1}$ to (a) gives
\begin{align*}
&(T_i^{-1}\otimes T_i^{-1})\Delta(T_i(y)) \\&=\sum_{n\geq0;h,h'\in H}c'(n;h,h')(-1)^nq_i^{n(n-1)}
    v_h\tJ_{|v_{h'}|-n\alpha_i^\vee}\tK_{|v_{h'}|-n\alpha_i^\vee}\otimes v_{h'}F_i^{(n)}\tK_{-n\alpha_i^\vee},
\end{align*}
where we have used the fact that $s_i(|u_{h'}|)=|v_{h'}|$.

Let $M=M(\lambda)$ be a Verma module, and let ${}^{\omega}M$ be the 
corresponding contragradient module with generator $\xi\in {}^\omega M$ 
satisfying $F_i\xi=0$.
Now, by Lemma \ref{C:TiandComult}, $(T_i^{-1}\otimes T_i^{-1})\Delta(T_i(y))=L_i\Delta(y)L_i^{-1}$ as maps
on ${}^\omega M\otimes {}^\omega M$. Equivalently, as maps we have
\begin{align*}
\bigg(\sum_{n\geq0;h,h'\in H}\hspace{-1em}&c'(n;h,h')(-1)^nq_i^{n(n-1)}v_h\tJ_{|v_{h'}|-n\alpha_i^\vee}\tK_{|v_{h'}|-n\alpha_i^\vee}\otimes v_{h'}F_i^{(n)}\tK_{-n\alpha_i^\vee}\bigg)L_i\\
\tag{c}    &=L_i\left(\sum_{n\geq0;h,h'\in H}d'(n;h,h')E_i^{(n)}v_h \tJ_{|v_{h'}|}\tK_{|v_{h'}|}\otimes v_{h'}\right).
\end{align*}
 Now we apply the equality (c) above to the vector $\xi\otimes\xi\in {}^\omega M\otimes {}^\omega M$.
Since $\xi\otimes \xi$ is fixed by $L_i$ and $F_i\xi=0$, the left-hand side becomes
\begin{align*}
&LHS=\sum_{n\geq0;h,h'\in H}c'(n;h,h')(-1)^nq_i^{n(n-1)}v_h\\
&\hspace{6em}\times (\tJ_{|v_{h'}|-n\alpha_i^\vee}\tK_{|v_{h'}|-n\alpha_i^\vee}\otimes v_{h'}F_i^{(n)}\tK_{-n\alpha_i^\vee})L_i (\xi\otimes \xi)\\
&=\sum_{n\geq0;h,h'\in H}c'(n;h,h')(-1)^nq_i^{n(n-1)}v_h\tJ_{|v_{h'}|-n\alpha_i^\vee}\tK_{|v_{h'}|-n\alpha_i^\vee}\xi\otimes v_{h'}F_i^{(n)}\tK_{-n\alpha_i^\vee}\xi\\
&=\sum_{h,h'\in H}c'(0;h,h')v_h\tJ_{|v_{h'}|}\tK_{|v_{h'}|}\xi\otimes v_{h'}\xi
\end{align*}
We also have that the right-hand side becomes
\begin{align*}
RHS&=L_i\left(\sum_{n\geq0;h,h'\in H}d'(n;h,h')E_i^{(n)}v_h \tJ_{|v_{h'}|}\tK_{|v_{h'}|}\otimes v_{h'}\right) (\xi\otimes \xi)\\
&=\sum_{n,t\geq0;h,h'\in H}d'(n;h,h')(-1)^t\pi_i^t(\pi_iq_i)^{t\choose 2}(\pi q-q^{-1})^t[t]^!_i\\
&\hspace{5em}\times F_i^{(t)}E_i^{(n)}v_h \tJ_{|v_{h'}|}\tK_{|v_{h'}|}\xi\otimes E_i^{(t)}v_{h'}\xi\tag{d}
\end{align*}
Let
$$\varpi:{}^\omega M\longrightarrow {}^\omega M/ E_i\, {}^\omega M$$
be the canonical projection. Applying $1\otimes\varpi$ to (d), we see that the right-hand side is nonzero in ${}^\omega M\otimes({}^\omega M/E_i\,{}^\omega M)$ only if $t=0$. Therefore, in ${}^\omega M\otimes({}^\omega M/E_i\,{}^\omega M)$,
\begin{align*}
\sum_{h,h'\in H}c'(n;h,h')E_i^{(n)}v_h\xi\otimes\varpi(v_{h'})=\sum_{h,h'\in H}d'(0;h,h')v_h\xi\otimes\varpi(v_{h'}\xi).
\end{align*}
Since ${}^\omega M$ is a free $\Up$-module, Corollary \ref{cor:direct sum UpJ}
implies that $E_i^{(n)}v_h\xi\in{}^\omega M$ are linearly independent for all
$n\geq 0$ and $h\in H$. In particular, we must have $$\sum_{h'\in H}(d'(0;h,h')-c'(0;h,h'))v_{h'}\xi\in E_i\,{}^\omega M$$ for each $h\in H$.
But, $\UpJ[i]\cap E_i\UpJ=0$, so we conclude $d'(0;h,h')=c'(0;h,h')$ for all $h,h'\in H$. This proves the result.
\end{proof}

\subsection{Computations with the Inner Product} Recall the inner product on $\UpJ$ that was defined in $\S$\ref{SS:UPJ}.

\begin{lemma}\label{L:innerproduct1} Assume that $m+m'=-a_{ij}$. Then,
\[\pi_i^{m\choose 2}(e(i,j;m),e(i,j;m))=\pi_i^{m'\choose 2}(e'(i,j;m),e'(i,j;m)).\]
\end{lemma}

\begin{proof}
By Proposition \ref{P:ffi}(a) and \eqref{eq:derivsInnerProd}, we have that $(e(i,j;m),E_i\UpJ )=0$. Thus we have
\begin{align*}
(e&(i,j;m),e(i,j;m))=(e(i,j;m),E_jE_i^{(m)})\\
    &=( r(e(i,j;m)),E_j\otimes E_i^{(m)})\\
    &=\prod_{h=0}^{m-t-1}(1-\pi_i^{h+1-m}q_i^{2h+2-2(m-a_{ij}}))(E_i^{(m)},E_i^{(m)})(E_j,E_j)\\
    &=\pi_i^{\binom{m}{2}}\prod_{h=1}^m\frac{1-\pi_i^{h-m}q_i^{2h+2m'}}{1-\pi_i^hq_i^{2h}}(E_j,E_j)\\
    &=\pi_i^{\binom{m}{2}}q_i^{mm'}\prod_{h=1}^m\frac{(\pi_i q_i)^{h+m'}-q_i^{-h-m'}}{\pi_i^hq_i^{h}-q_i^{-h}}(E_j,E_j)\\
    &=\pi_i^{\binom{m}{2}}q_i^{mm'}\left[{m+m'\atop m}\right]_i(E_j,E_j)
\end{align*}
Similarly, we compute
\begin{equation*}
(e'(i,j;m'),e'(i,j;m'))=\pi_i^{m'\choose 2}q_i^{m'm}\left[{m+m'\atop m'}\right]_i(E_j,E_j).
\end{equation*}
The lemma follows.
\end{proof}

\begin{proposition}\label{P:Tiinvariance}
For any $x,y\in \UpJ[i]$ we have
$$\pi^{\binom{|T_i^{-1}(x)|}{2}}(T_i^{-1}(x),T_i^{-1}(y))=\pi^{\binom{|x|}{2}}(x,y),$$
where, for each $\nu=\sum \nu_i\alpha_i\in Q$, $\binom{\nu}{2}=\sum \binom{\nu_i}{2}d_i$.
\end{proposition}

\begin{proof} Indeed, assume $x',x''\in\UpJ$ are such that
$$\pi^{\binom{|T_i^{-1}(x')|}{2}}(T_i^{-1}(x'),T_i^{-1}(y'))=\pi^{\binom{|x'|}{2}}(x',y'),$$
and
$$\pi^{\binom{|T_i^{-1}(x'')|}{2}}(T_i^{-1}(x''),T_i^{-1}(y''))=\pi^{\binom{|x''|}{2}}(x'',y''),$$
for any $y',y''\in\UpJ[i]$. We show that
$$\pi^{\binom{|T_i^{-1}(x'x'')|}{2}}(T_i^{-1}(x''),T_i^{-1}(y))=\pi^{\binom{|x'x''|}{2}}(x'',y),$$
for any $y\in\UpJ[i]$.

By definition, we have
$$(x'x'',y)=(x'\otimes x'',r(y))$$
and
$$(T_i^{-1}(x'x''),T_i^{-1}(y))=(T_i^{-1}(x')\otimes T_i^{-1}(x''),r(T_i^{-1}(y))).$$
We have $(x',\UpJ E_i)=0$ since $\ri(x')=0$, and so
$$(x'\otimes x'',r(y))=(x'\otimes x'',{'r}(y)).$$
Also, $(T_i^{-1}(x''),E_i\UpJ)=0$ since $\ir(T_i^{-1}(x''))=0$, hence
\begin{align*}
(T_i^{-1}(x')\otimes T_i^{-1}(x''),r(T_i^{-1}(y))&=(T_i^{-1}(x')\otimes T_i^{-1}(x''),{''r}(T_i^{-1}(y))\\
	&=(T_i^{-1}(x')\otimes T_i^{-1}(x''),(T_i^{-1}\otimes T_i^{-1}){'r}((y)),
\end{align*}
which follows from Lemma \ref{L:ffi6}. Write $'r(y)=\sum y_{(1)}\otimes y_{(2}$. Then,
$$(x'\otimes x'',r(y))=\sum(x',y_{(1)})(x'',y_{(2)}).$$
Since $T_i^{-1}$ is an even algebra homomorphism on $\UU$, we have
$$T_i^{-1}\otimes T_i^{-1}\parens{\sum y_{(1)}\otimes y_{(2)}}=\sum T_i^{-1}(y_{(1)})\otimes T_i^{-1}(y_{(2)}).$$
In particular, we see that
$$(T_i^{-1}(x')\otimes T_i^{-1}(x''),r(T_i^{-1}(y))=\sum(T_i^{-1}(x'),T_i^{-1}(y_{(1)})) (T_i^{-1}(x''),T_i^{-1}(y_{(2)})).$$
Now, for all $z\in\Up_J[i]$, if
$$|z|=\sum_{j\in I}\nu_j\af_j\andeqn |T_i^{-1}(z)|=\sum_{j\in I}\nu_j'\af_j,$$
then $\nu_j p(j)\equiv \nu_j'p(j)$ (mod 2) for all $j\in I$. Hence,
\begin{align*}
\pi^{{|x'|\choose 2}+{|x''|\choose 2}}\pi^{{|T_i^{-1}(x')|\choose 2}+{|T_i^{-1}(x'')|\choose 2}}
&=\pi^{{|x'|+|x''|\choose 2}}\pi^{{|T_i^{-1}(x')|+|T_i^{-1}(x'')|\choose 2}}\\
&=\pi^{{|x'x''|\choose 2}}\pi^{{|T_i^{-1}(x'x'')|\choose 2}}.
\end{align*}
Then by the induction hypothesis,
\begin{align*}
(T_i^{-1}(x'x''),T_i^{-1}(y))&=\sum(T_i^{-1}(x'),T_i^{-1}(y_{(1)})) (T_i^{-1}(x''),T_i^{-1}(y_{(2)}))\\
&=\pi^{{|x'|\choose 2}+{|x''|\choose 2}+{|T_i^{-1}(x')|\choose 2}+{|T_i^{-1}(x'')|\choose 2}}\sum(x',y_{(1)}) (x'',y_{(2)})\\
&=\pi^{{|x'x''|\choose 2}}\pi^{{|T_i^{-1}(x'x'')|\choose 2}}\sum(x',y_{(1)}) (x'',y_{(2)})\\
&=\pi^{{|x'x''|\choose 2}}\pi^{{|T_i^{-1}(x'x'')|\choose 2}} (x'x'',y).
\end{align*}

Finally, we have reduced to checking that the proposition holds for $x$ a generator of $\UpJ[i]$ (i.e. $x=e(i,j;m)$). We may assume that $y$ is homogeneous of the same weight as $x$. Since $y\in\UpJ[i]$, this forces $y$ to be a scalar multiple of $e(i,j;m)$. Therefore, the proposition follows from Lemma \ref{L:innerproduct1}.
\end{proof}

\begin{definition}\label{D:addmissible} A sequence $\bh=(i_1,\ldots,i_n)\in I^n$ is said to be \textbf{admissible} if, for any $1\leq a\leq b\leq n$,
\begin{enumerate}
\item[(a)] $T_{i_a}T_{i_{a+1}}\cdots T_{i_{b-1}}(E_{i_b})\in \UpJ$, and
\item[(b)] $T_{i_b}^{-1}T_{i_{b-1}}^{-1}\cdots T_{i_{a+1}}^{-1}(E_{i_a})\in\UpJ$.
\end{enumerate}
Now, assume $\bh$ is admissible, and $1\leq p\leq n$. We say that $x\in\UpJ$ is \textbf{adapted} to $(\bh,p)$ if,
\begin{enumerate}
\item[(c)] $T_{i_a}T_{i_{a+1}}\cdots T_{i_p}(x)\in\UpJ$, for any $1\leq a\leq p$, and
\item[(d)] $T_{i_b}^{-1}T_{i_{b-1}}^{-1}\cdots T_{i_{p+1}}^{-1}(x)\in\UpJ$, for any $p+1\leq b\leq n$.
\end{enumerate}
Given $x\in\UpJ$ adapted to $(\bh,p)$ as above, and a sequence $\bc=(c_1,\ldots,c_n)\in\N_n$, define
\begin{align*}
L(\bh,\bc,p,x)=&E_{i_{p+1}}^{(c_{p+1})}\cdot[T_{i_{p+1}}(E_{i_{p+2}}^{(c_{p+2})})]\cdots[T_{i_{p+1}}T_{i_{p+2}}\cdots T_{i_{n-1}}(E_{i_n}^{(c_n)})]\\
	&\cdot\, x \,\cdot[T_{i_p}^{-1}T_{i_{p-1}}^{-1}\cdots T_{i_2}^{-1}(E_{i_1}^{(c_1)})]\cdots[T_{i_p}^{-1}(E_{i_{p-1}}^{(c_{p-1})})]\cdot E_{i_p}^{(c_p)}.
\end{align*}
Then, by definition, $L(\bh,\bc,p,x)\in\UpJ$.
\end{definition}

\begin{proposition}\label{prop:admissibleinnerprod}
Let $\bc=(c_1,c_2,\ldots, c_n), \bc'=(c_1',c_2',\ldots, c_n')\in \N^n$. Let $\bh\in I^n$ be admissible and suppose
$x,x'\in \UpJ$ is adapted to $(\bh,p)$ for some $1\leq p \leq n$. Then there exists $\ell(\bh,\bc, p, x)\in \Z$
such that
\[(L(\bh,\bc,p,x),L(\bh,\bc', p, x'))=\pi^{\ell(\bh,\bc, p, x)}(x,x')\prod_{s=1}^n (E_{i_s}^{(c_s)}, E_{i_s}^{(c'_s)}).\]
\end{proposition}
\begin{proof}
For any $i\in I$, $t,t'\in \N$ and $y,y'\in \UpJ[i]$, we have
\[\tag{$\bigstar$}(E_i^{(t)}y, E_i^{(t')}y')=(E_i^{(t)},E_i^{(t')})(y,y').\]
Similarly, if $z,z'\in {}^\sigma\UpJ[i]$, we have
\[\tag{$\bigstar\bigstar$}(zE_i^{(t)}, z'E_i^{(t')})=(E_i^{(t)},E_i^{(t')})(z,z').\]
Suppose $p<n$ and the proposition holds for $p+1$. Let $\tilde\bc,\tilde\bc'$ be the sequences
defined by $\tilde c_{p+1}=\tilde c_{p+1}'=0$, $\tilde c_{s}= c_s$ and $\tilde{c_{s}}'=c_{s}'$ for $s\neq p+1$.
Let $\tilde x=T^{-1}_{i_{p+1}}(x)$, $\tilde x'=T^{-1}_{i_{p+1}}(x')$, and $\tilde p=p+1$.
Then $\tilde x$ is adapted to $(\bh, \tilde p)$ and
\[L(\bh,\bc,p,x)=E_{i_{p+1}}^{(c_{p+1})}T_{i_{p+1}}(L(\bh,\tilde \bc,\tilde p,\tilde x)).\]
By assumption, we have $T_{i_{p+1}}(L(\bh,\tilde \bc,\tilde p,\tilde x))\in \UpJ$. In particular, this implies that
$T_{i_{p+1}}(L(\bh,\tilde \bc,\tilde p,\tilde x))\in \UpJ[i_{p+1}]$ and $L(\bh,\tilde \bc,\tilde p,\tilde x)\in {}^\sigma\UpJ[i_{p+1}]$.
Similarly, \[L(\bh,\bc',p,x)=E_{i_{ p+1}}^{(c'_{p+1})}T_{i_{p+1}}(L(\bh,\tilde \bc',\tilde p,\tilde x))
\quad\text{ and }\quad L(\bh,\tilde \bc',\tilde p,\tilde x)\in {}^\sigma\UpJ[i_{p+1}].\]
Let $\nu=|L(\bh,\tilde\bc,\tilde p,\tilde x)|$. Then using $(\bigstar)$, Proposition \ref{P:Tiinvariance}, and the induction hypothesis,
we see that
\begin{align*}
(L(\bh,\bc,p,x),&L(\bh,\bc',p,x))\\
&=(E_{i_{p+1}}^{(c_{p+1})},E_{i_{p+1}}^{(c_{p+1}')})(T_{i_{p+1}}(L(\bh,\tilde\bc,\tilde p,\tilde x)),T_{i_{p+1}}(L(\bh,\tilde \bc',\tilde p,\tilde x)))\\
&=\pi^{{\nu\choose 2}+{s_i(\nu)\choose 2}}(E_{i_{p+1}}^{(c_{p+1})},E_{i_{p+1}}^{(c_{p+1}')})(L(\bh,\tilde\bc,\tilde p,\tilde x),L(\bh,\tilde \bc',\tilde p,\tilde x))\\
&=\pi^{\ell(\bh,\bc, p, x)+{\nu\choose 2}+{s_{i_{p+1}}(\nu)\choose 2}}(x,x')\prod_{s=1}^n (E_{i_s}^{(c_s)}, E_{i_s}^{(c'_s)})
\end{align*}
Therefore, it suffices to assume $p=n$, whence
\[L(\bh,\bc,n,x)=x \,\cdot[T_{i_n}^{-1}T_{i_{n-1}}^{-1}\cdots T_{i_2}^{-1}(E_{i_1}^{(c_1)})]\cdots[T_{i_n}^{-1}(E_{i_{n-1}}^{(c_{n-1})})]\cdot E_{i_n}^{(c_n)}.\]
When $n=0$, the result is trivial. Now assume $n>0$ and suppose the result holds for $n-1$. Let $\tilde x=T_{i_n}(x)$,
$T_i(x')=\tilde x'$, $\tilde\bh=(i_1,\ldots, i_{n-1})$, and $\tilde\bc=(c_1,\ldots, c_{n-1})$. Then
\[L(\bh,\bc,n,x)=T_{i_n}^{-1}(L(\tilde\bh,\tilde\bc,n-1,\tilde x)) E_{i_n}^{(c_n)},\]
\[L(\bh,\bc',n,x')=T_{i_n}^{-1}(L(\tilde\bh,\tilde\bc',n-1,\tilde x')) E_{i_n}^{(c'_n)}.\]
Then as before, we apply ($\bigstar\bigstar$), Proposition \ref{P:Tiinvariance}, and the induction hypothesis to obtain
\begin{align*}
(L(\bh,\bc,n,x),&L(\bh,\bc',n,x))\\
&=(E_{i_n}^{(c_{i_n})},E_{i_n}^{(c_{i_n}')})(T_{i_{n}}(L(\bh,\tilde\bc, n-1,\tilde x)),T_{i_{n}}(L(\bh,\tilde \bc',n-1,\tilde x')))\\
&=\pi^{\ell(\bh,\bc, n-1, x)+{\nu\choose 2}+{s_{i_n}(\nu)\choose 2}}(x,x')\prod_{s=1}^n (E_{i_s}^{(c_s)}, E_{i_s}^{(c'_s)}),
\end{align*}
where $\nu=|L(\tilde\bh,\tilde\bc,n-1,\tilde x)|$. This finishes the proof.
\end{proof}

\section{Braid group relations}

\subsection{The rank 2 PBW basis}\label{ss:r2PBW}
In this section, we assume $|I|=2$ and that $[a_{ij}]_{i,j\in I}$ is of finite type.
\begin{lemma}\label{L:BraidRelationsRank2}
Let $i,j\in I$, $i\neq j$. Then, as
automorphisms of $\Uint$,
$$\underbrace{T_iT_jT_i\cdots}_{m_{ij}}=\underbrace{T_jT_iT_j\cdots}_{m_{ij}}.$$
\end{lemma}
\begin{proof}
We assume $m_{ij}\in\{2,3,4,6\}$ as otherwise there is nothing to
prove. Moreover, when both $i,j\in I_\zero$, this is \cite[Section
39.2]{Lu}. We may, therefore, assume that either $i$ or $j$ is odd.
Then we must have
$m_{ij}\in\{2,4\}$ and, if both $i,j\in I_\one$, then $m_{ij}=2$.

First, assume that $m_{ij}=2$, so $\ang{i,j'}=\ang{j,i'}=0$. Then
$T_j(E_i)=\tJ_j^{p(i)}E_i$ and $T_j(F_i)=\tJ_j^{p(i)}F_i$.
Therefore, we have
\begin{align*}
T_iT_j(E_i)&=T_i(\tJ_j^{p(i)}E_i)
 =-\pi_i\tJ_j^{p(i)}\tJ_i\tK_i^{-1}F_i\\
T_jT_i(E_i)&=T_j(-\pi_i\tJ_i\tK_i^{-1}F_i)
 =-\pi_i\tJ_i\tK_i^{-1}\tJ_j^{p(i)}F_i
\end{align*}
so $T_iT_j(E_i)=T_jT_i(E_i)$. By symmetry $T_iT_j(E_j)=T_jT_i(E_j)$.
By a similar argument, we deduce that $T_iT_j(F_i)=T_jT_i(F_i)$ and
$T_iT_j(F_j)=T_jT_i(F_j)$. It is clear that
$T_iT_j(K_\mu)=T_jT_i(K_\mu)$. Therefore, the lemma holds in
this case.

Now, Assume that $m_{ij}=4$. We may assume $i\in I_\one$ and $j\in
I_\zero$, so that $\ang{i,j'}=-2$ and $\ang{j,i'}=-1$. In this case,
we have $q_j=q_i^2$. Additionally, we will repeatedly use the fact
that
\begin{align}\tag{a}
\pi_i^{p(j)}=\pi_j^{p(i)}=1, \qquad
\tJ_i^{p(j)}=\tJ_j^{p(i)}=1.
\end{align}

$$
\xy
(0,2)*{\bullet};(10,2)*{\circ}**\dir{=};(5,2)*{<};(0,-1)*{\scriptsize
i};(10,-1)*{\scriptsize j};\endxy
$$

Recall the elements $e_{1,m}=e(i,j;m)$ and
$e'_{1,m}=e'(i,j;m)$. Since
$\sigma(e_{1,m})=e'_{1,m}$, Lemma \ref{L:HigherSerre} implies that
\begin{align}\tag{b}
-q_i^{2-2m}\pi_i^{m+1}e'_{1,m}E_i+E_ie'_{1,m}=[m+1]_ie_{1,m+1}.
\end{align}
Using Theorem \ref{T:BraidingOnU}, we have that
$T_i(E_j)=\pi_ie_{1,2}$ and $T_i^{-1}(E_j)=\pi_ie'_{1,2}$.
Interchanging the roles of $i$ and $j$ in Theorem
\ref{T:BraidingOnU} and using the relations $q_j=q_i^2$ and (a), we
have
\begin{align*}
T_j(E_i)&=e_{j,i,1,1}=E_iE_j-q_jE_jE_i\\
    &=E_iE_j-q_i^2E_jE_i=e'_{i,j;1,1}=e'_{1,1},
\end{align*}
and similarly $T_j^{-1}(E_i)=e'_{j,i,1,1}=e_{i,j;1,1}=e_{1,1}$.
Therefore, using (b) we have
\begin{align*}
T_j^{-1}(e'_{1,2})&=[2]_i^{-1}T_j^{-1}(-e'_{1,1}E_i+E_ie'_{1,1})\\
    &=[2]_i^{-1}T_j^{-1}(-e'_{1,1})T_j^{-1}(E_i)+T_j^{-1}(E_i)T_j^{-1}(e'_{1,1})\\
    &=[2]_i^{-1}(-E_ie_{1,1}+e_{1,1}E_i)=e_{1,2}.
\end{align*}
It follows that $T_j(e_{1,2})=e'_{1,2}$. By Lemma
\ref{L:Braidsones} and the fact that $p(j)=0$, $T_i(e_{1,1}')=\pi_ie_{1,1}$ and, therefore,
$$\xymatrix{E_j\ar@{|->}[r]^{T_i}&\pi_ie_{1,2}\ar@{|->}[r]^{T_j}&\pi_ie'_{1,2}\ar@{|->}[r]^{T_i}&E_j}$$
and
$$\xymatrix{E_i\ar@{|->}[r]^{T_j}
 &e'_{1,1}\ar@{|->}[r]^{T_i}&\pi_ie_{1,2}\ar@{|->}[r]^{T_j}&\tJ_iE_i}$$
By a similar computation,
$$T_iT_jT_i(F_j)=F_j\;\;\;\mbox{and}\;\;\;T_jT_iT_j(F_i)=\pi_iF_i.$$
Hence,
\begin{align*}
T_jT_iT_jT_i(E_j)&=T_j(E_j)=-\tJ_j\tK_j^{-1}F_j, \\
T_iT_jT_iT_j(E_j)&=T_iT_jT_i(-\tJ_j\tK_j^{-1}F_j)=-\tJ_j\tK_j^{-1}F_j, \\
T_jT_iT_jT_i(E_i)&=T_jT_iT_j(-\pi_i\tJ_i\tK_i^{-1}F_i)=-\tJ_i\tK_i^{-1}F_i, \\
T_iT_jT_iT_j(E_i)&=T_i(\tJ_iE_i)=-\tJ_i\tK_i^{-1}F_i,
\end{align*}
where we have used that $s_is_js_i(j)=j$ in the
second line, and $s_js_is_j(i)=i$ in the third.
Therefore, $T_iT_jT_iT_j$ and $T_jT_iT_jT_i$ agree on $E_i$ and
$E_j$. By a similar argument, they agree on $F_i$ and $F_j$. It is
easy to prove that they agree on $K_\mu$ and $J_\mu$, therefore, they
are equal. This proves the theorem.
\end{proof}

Now let $m=m_{ij}$. Then for any $p$ such that $0\leq p \leq m-1$,
\[(\underbrace{\ldots T_jT_i}_{p\text{ factors}})(E_j),\ \ (\underbrace{\ldots T_iT_j}_{p\text{ factors}})(E_i),
\ \ (\underbrace{\ldots T_j^{-1}T_i^{-1}}_{p\text{ factors}})(E_j),\ \ (\underbrace{\ldots T_i^{-1}T_j^{-1}}_{p\text{ factors}})(E_i)\in \UpJ.\]
In particular, the sequences $\bi=(i,j,i,j,\ldots)$ and $\bj=(j,i,j,i,\ldots)$ with $m$ terms are admissible sequences.
Consider the following sets of elements of $\UpJ$, where each element is a product of $m$ elements of $\UpJ$
and $\bc=(c_1,\ldots, c_m)\in \N^m$:
\[\tag{a}\set{E_i^{(c_1)} T_i(E_j^{(c_2)})T_iT_j(E_i^{(c_3)})\ldots\mid (c_1,\ldots, c_m)\in \N^m};\]
\[\tag{b}\set{E_j^{(c_1)} T_j(E_i^{(c_2)})T_jT_i(E_j^{(c_3)})\ldots\mid (c_1,\ldots, c_m)\in \N^m};\]
\[\tag{c}\set{E_i^{(c_1)} T_i^{-1}(E_j^{(c_2)})T_i^{-1}T_j^{-1}(E_i^{(c_3)})\ldots\mid (c_1,\ldots, c_m)\in \N^m};\]
\[\tag{d}\set{E_j^{(c_1)} T_j^{-1}(E_i^{(c_2)})T_j^{-1}T_i^{-1}(E_j^{(c_3)})\ldots\mid (c_1,\ldots, c_m)\in \N^m}.\]
Note that each set consists of elements of the form $\sigma^e(L(\bh,\bc,p,1))$ where $\bh=\bi$ or $\bj$,
$p=0$ or $m$, and $e=0$ or $1$. In particular, by Proposition \ref{prop:admissibleinnerprod} each set consists of
pairwise orthogonal elements of $\UpJ$; in addition, if $x$ is an element of one of these sets, then
$(x,x)$ is not a zero divisor in $\UzJ$, and therefore each set is linearly independent.

\begin{lemma}
Each of the sets (a)-(d) is a basis of the free $\UzJ$-module $\UpJ$.
\end{lemma}
\begin{proof}
Because the characters of $\Up$ and $\Up|_{\pi=1}$ are the same, the proof of this fact is identical to the proof of
\cite[Lemma 39.3.2]{Lu}.
\end{proof}

\subsection{Proof of the braid relations on modules}

Recall that we denote the highest weight vector of $V(\lambda)$ by $\eta_\lambda$.

\begin{lemma}\label{lem:redexpandbraids}
Let $\bh=(i_1,\ldots, i_N)$ be a sequence in $I$ such that $s_{i_1}\ldots s_{i_N}$ is a reduced expression in $W$. Let $\lambda\in P_+$ and
$a_k=\ang{s_{i_N}\ldots s_{i_{k+1}}(\alpha_{i_k}^\vee),\lambda}$. Then
\[T_{i_1}\ldots T_{i_N}\eta_\lambda =F_{i_1}^{(a_1)}\ldots F_{i_N}^{(a_N)} \eta_\lambda.\]
\end{lemma}
\begin{proof}
Note that this is trivially true when $N=0$, and that $N=1$ follows from Lemma \ref{L:BraidGroupOnModules}.
Now assume $N\geq 2$ and let $\eta(\bh)=T_{i_1}\ldots T_{i_N}\eta_\lambda$. Then by induction,
it suffices to show that $T_{i_1}\eta(\bh')=F_{i_1}^{(a_1)}\eta(\bh')$ where $\bh'=(i_1,\ldots, i_{N-1})$.

Let $\mu=s_{i_2}\ldots s_{i_N}(\lambda)$. Note that $\eta(\bh')\in V(\lambda)_{\mu}$ and
$\ang{\alpha_{i_1}^\vee,\mu}=a_1$ by the $W$-invariance of $\ang{-,-}$. In particular, if
$E_{i_1} \eta(\bh)=0$ then $T_{i_1}\eta(\bh')=F_{i_1}^{(a_1)}\eta(\bh')$ by Lemma \ref{L:BraidGroupOnModules}.
Therefore, it remains to show that $E_{i_1} \eta(\bh)=0$.

Now note that $E_{i_1}\eta(\bh')\in V(\lambda)_{\mu+\alpha_{i_1}}$, so it suffices to show this weight space is zero.
Assume to the contrary that $V(\lambda)_{\mu+\alpha_{i_1}}\neq 0$. Then since
$s_{i_2}\ldots s_{i_N}(\mu+\alpha_{i_1})=\lambda+s_{i_N}\ldots s_{i_2}(\alpha_{i_1})$,
we have that $V(\lambda)_{\lambda+s_{i_N}\ldots s_{i_2}(\alpha_{i_1})}\neq 0$.
But then $s_{i_N}\ldots s_{i_2}(\alpha_{i_1})<0$, which contradicts that $s_{i_1}\ldots s_{i_N}$ is a reduced expression.
This completes the proof.
\end{proof}

\begin{proposition}[Quantum Verma Identity]\label{prop:QVI} Assume that $|I|=2$, $[a_{ij}]_{i,j\in I}$ is of finite type, and $p(i)p(j)=0$.
Let $\lambda\in P_+$. Define \[a_k=\langle\underbrace{\ldots s_js_is_j}_{m-k\text{ factors}}(\alpha_i^\vee),\lambda\rangle,
\qquad b_k=\langle\underbrace{\ldots s_is_js_i}_{m-k\text{ factors}}(\alpha_j^\vee),\lambda\rangle.\]
Set $x=F_i^{(a_1)}F_j^{(a_2)}F_i^{(a_3)}\ldots$ and $y=F_j^{(b_1)}F_i^{(b_2)}F_j^{(b_3)}\ldots$
where both products have $m$ factors. Then $x=y$.
\end{proposition}
\begin{proof}
If $i,j\in I_{\zero}$, then the statement of the proposition follows from \cite[Proposition 39.3.7]{Lu}.
If $\ang{i,j'}=0$, then since $p(i)p(j)=0$, the statement is trivially true
by the Serre relation $F_iF_j=F_jF_i$.
Therefore, we may assume $i\in I_{\one}$ and $m=4$. In this case, a similar proof to Lusztig's can be given,
however, we will sketch a shorter proof here by utilizing the theory of twistors from \cite{CFLW}.

By direct computation we see that $x,y\in \UU_\nu^-$ where
$\nu=2\ang{\alpha_i^\vee+\alpha_j^\vee,\lambda}\alpha_{i}
+\ang{\alpha_i^\vee+2\alpha_j^\vee,\lambda}\alpha_j$.
Moreover, $a_t=b_{5-t}$ and so $x=\varrho(y)$, where $\varrho=\omega^{-1}\sigma\omega$.
Now set $x=z^-$, where $z\in\ff$. Then we want to show $z=\varrho(z)$, where we define $\varrho:\ff\longrightarrow\ff$ by $\varrho(z_1)^-=\varrho(z_1^-)$ for any $z_1\in \ff$.

Let $\ff|_{\pi=\pm 1}$ denote the quotient of $\ff$ by the two-sided ideal generated by $\pi\mp 1$; in particular, note that
$\ff=(1+\pi)\ff\oplus(1-\pi)\ff=\ff|_{\pi=1}\oplus \ff|_{\pi=-1}$ as algebras. Let $(-)|_{\pi=\pm 1}$ be the canonical projections
and note that $\ff|_{\pi=1}$ is identically Lusztig's half quantum group. In particular,  \cite[Proposition 39.3.7]{Lu} implies $\varrho(z|_{\pi=1})=z|_{\pi=1}$,
so it suffices to prove that $\varrho(z|_{\pi=-1})=z|_{\pi=-1}$.

Let $\mathbf t^2=-1$. By \cite[Theorem 2.4]{CFLW}, there exists a $\Q(\mathbf t)$-linear bijection $\mathfrak X$ between  (scalar extensions of) $\ff|_{\pi=1}$ and $\ff|_{\pi=-1}$.
In particular, \[\mathfrak X(z|_{\pi=1})=\mathbf t^n z|_{\pi=-1}\] for some $n\in \Z$. Utilizing Proposition 2.6 of {\em loc. cit.},
we have $\mathfrak X(\varrho(z|_{\pi=1}))=(-1)^{n'} \mathbf t^n \varrho(z|_{\pi=-1})$
for some $n'\in \Z$. However, there is an explicit formula for $n'$ depending on $|z|=\nu$, and it can be computed directly that $n'\in 2\Z$. On the other hand,
\[\mathbf t^n z|_{\pi=-1}=\mathfrak X(z|_{\pi=1})=\mathfrak X(\varrho(z|_{\pi=1}))=\mathbf t^n \varrho(z|_{\pi=-1}),\]
and hence $z|_{\pi=-1}=\varrho(z|_{\pi=-1})$ as desired.
\end{proof}

Consider the case $I=I_{\one}=\set{i,j}$ and $\ang{\af_i^\vee,\af_j}=0$.
We note that Proposition \ref{prop:QVI} is {\bf not} true in this case. Indeed,
if $\ang{\alpha_i^\vee,\alpha_j}=0$ then we have the Serre relation $F_iF_j=\pi F_jF_i$, and so in general we have the identity \begin{equation}\label{eq:skewcommute}
F_i^{(a)}F_j^{(b)}=\pi^{ab}F_j^{(b)}F_i^{(a)}.
\end{equation}

\begin{lemma} Assume that $|I|=2$ and that $[a_{ij}]_{i,j\in I}$ is of finite type.
Let $M$ be an integrable $\UU$-module.
\begin{enumerate}
 \item Assume $p(i)p(j)=0$. Then we have
\[T_iT_jT_i\ldots=T_jT_iT_j\ldots:M\rightarrow M,\]
where both products have $m$ factors.
\item Assume $p(i)=p(j)=1$, so that $\ang{\alpha_i^\vee,\alpha_j}=\ang{\alpha_j^\vee,\alpha_i}=0$, $m=2$, and $P$ can be identified with $\Z\times \Z$. Then for $s,t\in \Z$,
\[T_iT_j=\pi^{st}T_jT_i:M_{s,t}\rightarrow M_{-s,-t}.\]
\end{enumerate}
\end{lemma}
\begin{proof}
The statement (1) is proved identically to \cite[Lemma 39.4.1]{Lu},
whereas (2) follows from a slightly modified proof. Indeed, assume $p(i)=p(j)=1$
and let us identify weights with $\Z\times \Z$ (where the first component
corresponds to $\alpha_i$, and the second corresponds to $\alpha_j$). Let
$x\in M_{s,t}$, and without loss of generality we may assume that $x=u\eta$,
where $E_i\eta=E_j\eta=0$. Suppose first that $u=1$. Then
by Lemma \ref{lem:redexpandbraids} we have $T_iT_j(\eta)=F_i^{(s)}F_j^{(t)}\eta$
and $T_jT_i(\eta)=F_j^{(t)}F_i^{(s)}\eta$. Then \eqref{eq:skewcommute}
implies $T_iT_j(\eta)=\pi^{st}T_jT_i(\eta)$.

Now suppose $u\in \UU_{c_i\alpha_i+c_j\alpha_j}$. Then $\eta\in M_{s-2c_i,t-2c_j}$,
so using Theorem \ref{T:BraidingOnU} and the previous case,
\[
T_iT_j(u\eta)=T_iT_j(u)T_iT_j(\eta)=T_jT_i(u)\pi^{(s-2c_i)(t-2c_j)}T_jT_i(\eta)=\pi^{st}T_jT_i(u\eta).\]
\end{proof}

Now we will drop the assumption $|I|=2$ and consider the general case.

\begin{theorem}\label{thm:braidgrouprelns}
Suppose that $i\neq j$ in $I$ such that $m=m_{i,j}<\infty$.
Let $M$ be an integrable $\UU$-module, $\lambda\in P$, and set
$\chi(\lambda)=\ang{\alpha_i^\vee,\lambda}\ang{\alpha_j^\vee,\lambda}$ and
$\lambda'=\ldots s_is_js_i(\lambda)=\ldots s_js_is_j(\lambda)$, where both products have $m$ factors.
Then we have the following equalities, where all products have $m$ factors:
\begin{enumerate}
\item $T_iT_jT_i\ldots=T_jT_iT_j\ldots:\UU\rightarrow \UU$;
\item $T_i^{-1}T_j^{-1}T_i^{-1}\ldots=T_j^{-1}T_i^{-1}T_j^{-1}\ldots:\UU\rightarrow \UU$;
\item $T_iT_jT_i\ldots=\pi^{\chi(\lambda)p(i)p(j)}T_jT_iT_j\ldots:M_{\lambda}\rightarrow M_{\lambda'}$.
\item $T_i^{-1}T_j^{-1}T_i^{-1}\ldots=\pi^{\chi(\lambda)p(i)p(j)}T_j^{-1}T_i^{-1}T_j^{-1}\ldots: M_{\lambda}\rightarrow M_{\lambda'}$.
\end{enumerate}
\end{theorem}

\begin{proof}
This is proved almost identically to \cite[Theorem 39.4.13]{Lu}, except for (1) in the case $p(i)=p(j)=1$. In this case, let $u\in \UU$ and set
$u_1=T_iT_jT_i\ldots(u)$ and $u_2=T_jT_iT_j\ldots(u)$. Take any integrable $\UU$-module $M$, and suppose $m\in M_\lambda$.
Set $\nu=|u|$. Since $i,j\in I_{\one}$, note that $\ang{\alpha_i^\vee,\nu'},\ang{\alpha_j^\vee,\nu'}\in 2\Z$, so in particular $\chi(\lambda+\nu)\equiv \chi(\lambda)$ modulo 2.
Then we have
\begin{align*}
u_1T_jT_iT_j\ldots(m)&=\pi^{\chi(\lambda)}u_1T_iT_jT_i\ldots(m)\\
&=\pi^{\chi(\lambda)}T_iT_jT_i\ldots(um)\\
&=\pi^{\chi(\lambda)+\chi(\lambda+\nu)}T_jT_iT_j\ldots(um)\\
&=u_2T_jT_iT_j\ldots(m).
\end{align*}
Then $u_1-u_2$ acts as $0$ on any integrable module $M$, and thus $u_1=u_2$
by \cite[Proposition 2.7.2]{CHW1}.
\end{proof}

As a result of Theorem \ref{thm:braidgrouprelns}, we see that $\UU$ carries
an action of the braid group $B$. In particular, for $w\in W$ we may define
$T_w=T_{i_1}\cdots T_{i_d} $ if
$w=s_{i_1}\cdots s_{i_d}$ is a reduced expression. As usual, we have
$T_{w_1}T_{w_2}=T_{w_1w_2}\mbox{ if
}\ell(w_1w_2)=\ell(w_1)+\ell(w_2).$ It follows by
Theorem~\ref{T:BraidingOnU}(c) that
$$
T_w(K_\mu)=K_{w(\mu)},\qquad T_w(J_\mu)=J_{w(\mu)}.
$$

For the integrable $\UU$-modules,
the situation is slightly more complicated.
Let $\lambda\in P$. The {\em spin} of the block $\catO_\lambda$  is a binary sequence
${\rm spin}(\lambda)\in \set{0,1}^{I}$ such that
$${\rm spin}(\lambda)_i\equiv\begin{cases}0&\mbox{if }i\in I_\zero,\\
     \ang{\alpha_i^\vee,\lambda}&\mbox{if }i\in I_\one\end{cases}\;\;\;\mbox{(mod 2)}.$$
Note that ${\rm spin}(\lambda+\nu)={\rm spin}(\lambda)$ for any
$\nu\in Q$ by condition (P1) on the GCM $A$. In particular, spin is an invariant of the block $\catO_\ld$.
We also
define the spin-parity function $p_{\ld}:I\rightarrow\set{0,1}$
via $p_{\ld}(i)={\rm spin}(\ld)_i$.

\begin{corollary}
Let $\ld\in P$ and $M\in \catO_{\ld}$.
Then the spin braid group $B(A,p_{\ld})$ acts on $M$.
\end{corollary}

\subsection{Reduced expressions and admissibility}

The braid operators can be used to inductively construct a PBW basis
for subspaces of $\UU$ using the approach in \cite[Chapter 40]{Lu}
almost without modification. For the readers convenience, we will
recall the essential results.

\begin{lemma}\label{lem:admissiblebraids}
Assume that $i\neq j\in I$ and let $m=m_{ij}\leq \infty$.
Let $p$ be an integer such that $0\leq p\leq m$.
Define the notations
\[T'_{i,j;p}=\underbrace{\ldots T_iT_jT_i}_{\rm p\ factors},\quad
T''_{i,j;p}=\underbrace{\ldots T^{-1}_iT^{-1}_jT^{-1}_i}_{\rm p\ factors}.\]
and let $\UU^+(i,j)$ be the $\UzJ$-subalgebra of $\UU$ generated by $E_i, E_j$.
Then $T'_{i,j;p}(E_j)$, $T''_{i,j;p}(E_j)\in \UU^+(i,j)$.
\end{lemma}

\begin{proof}
If $m<\infty$, then the statement follows from the explicit calculations
in the proof of Lemma \ref{L:BraidRelationsRank2}. In the case $m=\infty$,
the proof is virtually identical to that of \cite[Lemma 40.1.1]{Lu},
and we omit the details.
\end{proof}

\begin{lemma}\label{lem:redexpadmissible}
Let $w=s_{i_1}\ldots s_{i_n}$ be a reduced expression in $W$.
Then we have that $T_{i_1}\ldots T_{i_{n-1}}(E_{i_n})$ and
$T_{i_1}^{-1}\ldots  T_{i_{n-1}}^{-1}(E_{i_n})$ are in $\UpJ$.
\end{lemma}

\begin{proof}
This is proved exactly as \cite[Lemmas 40.1.2, 40.1.3]{Lu} using
Lemma \ref{lem:admissiblebraids}.
\end{proof}

\begin{proposition}
Let $w\in W$ and $\bh=(i_1,\ldots, i_n)$ be a sequence in $I$
such that $w=s_{i_1}\ldots s_{i_n}$ is a reduced expression.
Then the following statements hold.
\begin{enumerate}
\item The sequence $\bh$ is admissible.
\item The elements
$E_{i_1}^{(c_1)}T_{i_1}(E_{i_2}^{(c_2)})\ldots T_{i_1}T_{i_2}\ldots T_{i_{n-1}}(E_{i_n}^{(c_n)})$ with $c_1,\ldots, c_n\in \N$
form a $\UzJ$-basis of a subspace $\UpJ(w)$ of $\UpJ$, and this subspace does not depend on the sequence $\bh$.
\item The elements
$E_{i_1}^{(c_1)}T^{-1}_{i_1}(E_{i_2}^{(c_2)})\ldots T^{-1}_{i_1}T^{-1}_{i_2}\ldots T^{-1}_{i_{n-1}}(E_{i_n}^{(c_n)})$
form a $\UzJ$-basis of the subspace $\UpJ(w)$.
\item If $i\in I$ such that $l(s_iw)<l(w)$, then $E_i\UpJ(w)\subset \UpJ(w)$.
\end{enumerate}
\end{proposition}

\begin{proof}
The proof of (1), (4), and the independence of $\UpJ(w)$ from the choice
of $\bh$ is proved exactly as in \cite[Lemma 40.2.1]{Lu}.
The linear independence of the elements in (2) and (3) is proved
exactly as in the rank 2 case. To wit, by part (1) and
Proposition \ref{prop:admissibleinnerprod}, these elements are pairwise orthogonal
and if $x=E_{i_1}^{(c_1)}T_{i_1}(E_{i_2}^{(c_2)})\ldots T_{i_1}T_{i_2}\ldots T_{i_{n-1}}(E_{i_n}^{(c_n)})$ then $(x,x)$ is not a zero divisor, and thus the elements are linearly independent.
\end{proof}

In particular, we obtain a basis when the Cartan datum is
of finite type as follows.

\begin{corollary}
Suppose the Cartan datum is of finite type and $w_0=s_{i_1}\ldots s_{i_n}$
is a reduced expression for the longest element of $W$. Then
the elements
\[\set{E_{i_1}^{(c_1)}T_{i_1}(E_{i_2}^{(c_2)})\ldots T_{i_1}T_{i_2}\ldots T_{i_{n-1}}(E_{i_n}^{(c_n)})\mid c_1,\ldots, c_n\in \N}\]
form a $\UzJ$-basis of $\UpJ$. Likewise, the elements
\[\set{E_{i_1}^{(c_1)}T^{-1}_{i_1}(E_{i_2}^{(c_2)})\ldots T^{-1}_{i_1}T^{-1}_{i_2}\ldots T^{-1}_{i_{n-1}}(E_{i_n}^{(c_n)})\mid c_1,\ldots, c_n\in \N}\]
 for various $c_1,\ldots, c_n\in \N$
form a $\UzJ$-basis of $\UpJ$.
\end{corollary}

\end{document}